\documentclass[11pt]{amsart}
\usepackage{amsxtra}
\usepackage[all]{xy}

\theoremstyle{plain}
\numberwithin{equation}{section}

\newtheorem{theorem}{Theorem}[section]
\newtheorem{lemma}[theorem]{Lemma}
\newtheorem{definition-lemma}[theorem]{Definition-Lemma}
\newtheorem{proposition}[theorem]{Proposition}
\newtheorem{corollary}[theorem]{Corollary}
\newtheorem{definition}[theorem]{Definition}
\newtheorem{example}[theorem]{Example}
\newtheorem{remark}[theorem]{Remark}

\newcommand{\st} [1]     {\scriptscriptstyle{{#1}}}
\newcommand{\rmap}       {\longrightarrow}

\newcommand{\Ker}        {{\mathrm {ker}}}

\newcommand{\Grd}        {\mathcal{G}}

\newcommand{\SP} [1]     {{\left\langle {{#1}} \right\rangle}}

\newcommand{\oLambda}   {\overline{\Lambda}}

\newcommand{\cX}        {\mathcal{X}}

\newcommand{\LF}       {\mathrm{Lie}} % Lie functor
\newcommand{\he}       {\widehat{e}}

\newcommand{\sour}        {\mathsf{s}}
\newcommand{\tar}         {{\mathsf{t}}}

\newcommand{\vl}       {\scriptscriptstyle{\vee}}

\newcommand{\Lie}        {\mathcal L}

\newcommand{\dd}     {\mathrm{d}}

\begin{document}
%%%%%%%%%%%%%%%%%%%%%%%%%%%%%%%%%%%%%%%%%%%%%%%%%%%%%%%%%%%%%%%%%%%%%%%%%%%
%%%%%%%%%%%%%%%%%%%%%%    Title    %%%%%%%%%%%%%%%%%%%%%%%%%%%%%%%%%%%%%%%%
\title[]
{Multiplicative forms at the infinitesimal level}

\author[]{Henrique Bursztyn and  Alejandro Cabrera}

\address{Instituto de Matem\'atica Pura e Aplicada,
Estrada Dona Castorina 110, Rio de Janeiro, 22460-320, Brasil }
\email{henrique@impa.br}

\address{Department of Mathematics,
University of Toronto, 40 St. George Street, Toronto, Ontario, M5S
2E4 Canada } \email{ale.cabrera@gmail.com \footnote{A.C.'s current
address:
    Departamento de Matem\'atica Aplicada, Instituto de Matemática, Universidade Federal do Rio de Janeiro, CEP 21941-909, Rio de Janeiro - RJ, Brazil
}}

%cortiz@impa.br}

\date{}
%\subjclass[2000]{Primary 20G15; Secondary 14M17, 53D17}

\maketitle

\begin{abstract}
We describe arbitrary multiplicative differential forms on Lie
groupoids infinitesimally, i.e., in terms of Lie algebroid data.
This description is based on the study of linear differential forms
on Lie algebroids and encompasses many known integration results
related to Poisson geometry. We also revisit multiplicative
multivector fields and their infinitesimal counterparts, drawing a
parallel between the two theories.
\end{abstract}

\tableofcontents

\section{Introduction}\label{sec:intro}

This paper is devoted to the study of multiplicative differential
forms on Lie groupoids, with focus on their infinitesimal
counterparts. Given a Lie groupoid $\Grd$ over a manifold $M$,
recall that a $k$-form $\omega \in \Omega^k(\Grd)$ is called
\textit{multiplicative} if $m^*\omega = \mathrm{pr}_1^*\omega +
\mathrm{pr}_2^* \omega$, where $m:\Grd^{(2)}=\Grd\times_M\Grd\to
\Grd$ is the groupoid multiplication, and $\mathrm{pr}_i:
\Grd^{(2)}\to \Grd$, $i=1,2$, are the natural projections. Our goal
is to characterize multiplicative forms on $\Grd$ solely in terms of
information from its Lie algebroid. We will also discuss the
analogous problem for multiplicative multivector fields.

Multiplicative differential forms and multivector fields on Lie
groupoids have been studied for over 20 years in a variety of
contexts. On Lie groups, multiplicative bivector fields came to
notice in the late 1980s when the concept of Poisson Lie group (see
e.g. \cite{LW} and references therein) was systematized.
Multiplicative 2-forms on Lie groupoids appeared around the same
time as the symplectic forms on symplectic groupoids \cite{CDW,W},
originally introduced as part of a quantization scheme for Poisson
manifolds. Poisson Lie groups and symplectic groupoids naturally led
to Poisson groupoids \cite{W2}, which gave further impetus to the
study of multiplicative structures. In particular, multiplicative
vector fields on Lie groupoids turn out to encompass classical
lifting processes of general interest in differential geometry, see
\cite{Mac-Xu}. There are further connections between multiplicative
2-forms and bivector fields and the theory of moment maps, found
e.g. in \cite{BC,bcwz,ILX,MiWe,xu}.

A central issue when considering  multiplicative geometrical
structures on Lie groupoids concerns their infinitesimal
description, i.e., their description in terms of Lie-algebroid data.
As usual in Lie theory, relating global and infinitesimal objects
involves suitable differentiation and integration procedures.
Important classes of examples of multiplicative structures and their
infinitesimal counterparts include Poisson Lie groups and Lie
bialgebras (see e.g. \cite{LW}), symplectic groupoids and Poisson
manifolds (see \cite{catfel,CDW}), and, more generally, the
correspondence between Poisson groupoids and Lie bialgebroids
\cite{Mac-Xu1,Mac-Xu2}. Dirac structures \cite{courant,SW} fit into
a similar picture, being the infinitesimal versions of certain
multiplicative 2-forms, see \cite{bcwz}. As observed in \cite{Cr},
generalized complex structures \cite{Gua,Hi} provide another class
of geometrical structures encoding infinitesimal information that
can be integrated to multiplicative structures on Lie groupoids. In
establishing these infinitesimal-global correspondences, different
methods have been employed, some based on the integrability of
Lie-algebroid morphisms, e.g. \cite{BCO,Mac-Xu1,Mac-Xu2}, others on
infinite-dimensional arguments, e.g. \cite{bcwz,catfel,CX,ILX}.
Although identifying the infinitesimal versions of multiplicative
forms is key in some of the aforementioned works, only special cases
of this problem have been considered. In this paper, we treated it
in full generality.

From yet another perspective, multiplicative differential forms
arise as constituents of the Bott-Schulman double complex of Lie
groupoids \cite{Bott} (see also \cite{AC} and references therein),
which computes the cohomology of their classifying spaces. So the
problem of understanding multiplicative forms infinitesimally may be
seen as part of the problem of finding infinitesimal models for the
cohomology of classifying spaces. This broader viewpoint is explored
in the recent work \cite{AC}, leading to results closely related to
ours; a comparison between them is also discussed in this paper.

Our approach to describe multiplicative forms infinitesimally starts
with the study of \textit{linear} differential forms on vector
bundles $A\to M$. We observe (Theorem~\ref{prop:struc}) that any
linear $k$-form on $A$ is equivalent to a pair $(\mu,\nu)$ of
vector-bundle maps $\mu: A\to \wedge^{k-1}T^*M$, $\nu :A\to
\wedge^kT^*M$, covering the identity on $M$. If $A$ carries a Lie
algebroid structure, with bracket $[\cdot,\cdot]$ and anchor $\rho$,
we say that the pair $(\mu,\nu)$ is an \textit{IM $k$-form} ($IM$
standing for \textit{infinitesimally multiplicative}) if the
following compatibility conditions are satisfied: for all $u,v \in
\Gamma(A)$,
\begin{enumerate}
\item $i_{\rho(u)}\mu(v) = - i_{\rho(v)}\mu(u)$,
\item $\mu([u,v]) = \Lie_{\rho(u)}\mu(v)-i_{\rho(v)}\dd\mu(u) -
i_{\rho(v)}\nu(u)$,
\item $\nu([u,v])= \Lie_{\rho(u)}\nu(v) -
i_{\rho(v)}\dd\nu(u).$
\end{enumerate}

We prove in Theorem~\ref{thm:1-1} that multiplicative $k$-forms on a
source-simply-connected Lie groupoid $\Grd$ over $M$ are in
one-to-one correspondence with IM $k$-forms on its Lie algebroid
$A\to M$. Concretely, the IM $k$-form $(\mu,\nu)$ associated with a
multiplicative $k$-form $\omega\in \Omega^k(\Grd)$ is defined by
\begin{align*}
\SP{\mu(u), X_1\wedge\ldots\wedge
X_{k-1}} & = \omega(u,X_1,\ldots,X_{k-1}),\\
\SP{\nu(u), X_1\wedge\ldots\wedge X_{k}} & =
\dd\omega(u,X_1,\ldots,X_{k}),
\end{align*}
where $X_i\in TM$, $i=1,\ldots,k$, and we view $M\subseteq \Grd$ and
$A\subseteq T\Grd|_M$.

A special class of IM-forms is obtained as follows. Any closed form
$\phi \in \Omega^{k+1}(M)$ determines a map $\nu: A\to
\wedge^kT^*M$, $\nu(u)=-i_{\rho(u)}\phi$, satisfying condition $(3)$
above. The IM $k$-forms $(\mu,\nu)$ with $\nu$ of this type are
referred to as \textit{IM $k$-forms relative to $\phi$}; they are
the infinitesimal versions of multiplicative $k$-forms satisfying
$$
\dd\omega = \sour^*\phi - \tar^*\phi,
$$
where $\sour$ and $\tar$ denote the groupoid source and target maps.
For $k=2$, IM forms relative to $\phi$ include $\phi$-twisted
Poisson and Dirac structures \cite{SW}, and our
Theorem~\ref{thm:1-1} recovers their known integrations
\cite{bcwz,CX}. For arbitrary $k$, IM forms relative to $\phi$ were
studied in \cite{AC} in connection with the Weil algebra of a Lie
algebroid. These and other examples are discussed in this paper.

The method we use to integrate IM forms on Lie algebroids to
multiplicative forms on Lie groupoids relies entirely on the known
correspondence between Lie-algebroid and Lie-groupoid morphisms
(Lie's second theorem for Lie algebroids); in particular, we do not
resort to the path spaces of \cite{catfel,CF,severa}, hence avoiding
infinite dimensional constructions. Although our method is inspired
by \cite{BCO,Mac-Xu1,Mac-Xu2}, it brings a technical difference in
that we represent differential $k$-forms on a manifold $N$ by
\textit{functions} $\oplus^k TN \to \mathbb{R}$ (as opposed to maps
$\oplus^{k-1} TN \to T^*N$); this small variation greatly simplifies
computations, so even when restricted to known situations, our
general proof seems more direct than existing ones. The integration
of IM forms is carried out in two steps: first, we show that an IM
$k$-form on a Lie algebroid $A\to M$ defines an element in
$\Omega^k(A)$ whose associated function $\oplus^k TA \to \mathbb{R}$
is a Lie-algebroid morphism; second, upon integration, one obtains a
groupoid morphism $\oplus^k T\Grd \to \mathbb{R}$ which defines a
multiplicative $k$-form\footnote{Forms
 on $\Grd$ are commonly viewed as sections $\Grd\to \wedge^k T^*\Grd$. The multiplicative ones, for $k=1$, are such that
 $\Grd\to T^*\Grd$ is a groupoid morphism. However, this viewpoint does not extend
 to
$k\geq 2$, as in this case $\wedge^k T^*\Grd$ inherits no canonical
groupoid structure in general; in contrast, $\oplus^kT\Grd$ is
always a groupoid and multiplicative forms are conveniently
described by groupoids morphisms $\oplus^kT\Grd\to \mathbb{R}$. An
analogous discussion holds for multivector fields.}.

%Here we use that, for a Lie algebroid $A$ (resp. Lie groupoid
%$\Grd$), $\oplus^k TA$ (resp. $\oplus^k T\Grd$) canonically inherits
%a Lie algebroid (resp. Lie groupoid) structure, see
%Sections~\ref{sec:alg} and \ref{sec:mult}; note, however, that this
%may not be the case for $\wedge^k TA$ (resp. $\wedge^k T\Grd$) when
%the base manifold $M$ is different from a point.

In the last part of the paper, we revisit multiplicative multivector
fields on Lie groupoids, as in \cite{ILX}. We show how the very same
techniques used to study multiplicative forms apply to the dual
situation of multivector fields, leading to an alternative
 proof of the universal lifting theorem of \cite{ILX} (not involving
path spaces) and drawing a clear parallel between the two theories.

As a final remark, we note that the results in this paper admit a
natural formulation in terms of graded geometry. Multiplicative
forms and multivector fields on a given Lie groupoid $\Grd$ may be
seen as multiplicative \textit{functions} on the associated
\textit{graded} Lie groupoids $T[1]\Grd$ and $T^*[1]\Grd$,
respectively. On ordinary Lie groupoids, the infinitesimal
counterpart of a multiplicative function is a Lie-algebroid cocycle.
The same holds at the graded level and, from this perspective, our
results consist in using the geometry of $T[1]\Grd$ and $T^*[1]\Grd$
to obtain concrete descriptions of their Lie-algebroid cocycles. For
example, using the the natural multiplicative vector field on
$T[1]\Grd$ (the de Rham differential on $\Grd$), one identifies its
Lie-algebroid cocycles with IM forms (see
Theorem~\ref{thm:mainresult}); for an analogous description of the
Lie-algebroid cocycles of the graded groupoid $T^*[1]\Grd$ (see
Theorem~\ref{thm:maindual}), ones uses its canonical multiplicative
symplectic structure (defined by the Schouten bracket on $\Grd$).
 We will not
elaborate on the supergeometric viewpoint in this paper, though it
makes our results more intuitive.

The paper is organized as follows. In Section \ref{sec:linear}, we
consider linear differential forms on vector bundles $A\to M$ and
establish their correspondence with pairs of vector-bundle maps
$(\mu,\nu)$, where $\mu:A\to \wedge^{k-1}T^*M$ and $\nu:A\to
\wedge^k T^*M$. In Section \ref{sec:alg}, we define IM $k$-forms on
Lie algebroids and prove a compatibility result with tangent Lie
algebroid structures (Theorem \ref{thm:mainresult}). Section
\ref{sec:mult} is devoted to Theorem~\ref{thm:1-1}, which is the
correspondence between IM forms on Lie algebroids and multiplicative
forms on Lie groupoids; we also discuss several special cases of
this result. Section \ref{sec:VE} explains the relationship between
Theorem~\ref{thm:1-1} and the Van Est isomorphism of \cite{AC}. In
Section \ref{sec:dual}, we revisit the theory of multivector fields
from \cite{ILX}.

\subsection*{Acknowledgments} We thank D. Iglesias Ponte and C. Ortiz for helpful
discussions related to this project. Cabrera thanks CNPq for
financial support and IMPA for its hospitality and stimulating
environment during the development of this work. Bursztyn's research
has been supported by CNPq and Faperj. The authors acknowledge the
anonymous referees for their many useful suggestions, especially
regarding improvements in the introduction.

\subsection*{Notation, conventions and identities}\label{subsec:ident}

For vector bundles $A\to M$ and $B\to M$ over the \textit{same base}
$M$, a vector-bundle map $\Psi: A\to B$ is always assumed to cover
the identity map on $M$, unless stated otherwise. We denote its
transpose, or dual, by $\Psi^t:B^*\to A^*$. We denote the $k$-fold
direct sum of a vector bundle $q_A: A\to M$ by $\oplus^k_M A$, or
simply $\oplus^k A$ if there is no risk of confusion. We may also
use the notation $\prod_{q_A}^k A$ if we want to be explicit about
the projection map $q_A$ (this is relevant when dealing with double
vector bundles).

For a Lie groupoid $\Grd$ over $M$, we usually denote its source and
target maps by $\sour$ and $\tar$. The set $\Grd^{(2)}\subset
\Grd\times \Grd$ of composable pairs is defined by the condition
$\sour(g)=\tar(h)$, and the multiplication is denoted by
$m:\Grd^{(2)}\to \Grd$, $m(g,h)=gh$. The unit map $\epsilon: M \to
\Grd$ is often used to identify $M$ with its image in $\Grd$. The
Lie algebroid of $\Grd$ is $A\Grd=\ker(T\sour)|_M$, with anchor
$T\tar|_A:A\to M$ and bracket induced by right-invariant vector
fields. For a Lie algebroid $A\to M$, we denote its anchor by
$\rho_{\st{A}}$ and bracket by $[\cdot,\cdot]_{\st{A}}$ (or simply
$\rho$ and $[\cdot,\cdot]$, if there is no risk of confusion).

We introduce some notation and collect some identities that will be
useful for later computations. If $U_1,\ldots, U_m$ are vector
fields on a manifold $M$, we set
\begin{equation}\label{eq:Iop}
I^U_{m,r}:= i_{U_m}\ldots i_{U_r}, \;\;\; r\leq m,
\end{equation}
where $i_U$ is the usual contraction. An inductive application of
Cartan's formula gives
\begin{equation}\label{eq:cartan}
I^U_{m,1} \dd = \sum_{l=1}^m(-1)^{l+1}I_{m,l+1}^U \Lie_{U_l}
I^U_{l-1,1} + (-1)^m \dd I_{m,1}^U,
\end{equation}
where $\dd$ denotes the de Rham differential and $\Lie_U$ is the Lie
derivative. Given another vector field $X$ and recalling the
commutator formula $i_{[X,U_l]}=\Lie_Xi_{U_l}-i_{U_l}\Lie_X$, we
obtain
\begin{equation}\label{eq:LX}
\Lie_X I^U_{m,1}=\sum_{l=1}^mI^U_{m,l+1} i_{[X,U_l]}I^U_{l-1,1} +
I^U_{m,1}\Lie_X.
\end{equation}
Given a differential form $\alpha$, we also have
\begin{equation}\label{eq:alpha}
I^U_{m,1}(\dd f\wedge \alpha)=\sum_{l=1}^m(-1)^{l+1}\dd
f(U_l)I^U_{m,l+1} I^U_{l-1,1}\alpha + (-1)^m \dd f \wedge
I^U_{m,1}\alpha.
\end{equation}

We often use Einstein's summation convention when there is no risk
of confusion.

%%%%%%%%%%%%%%%%%%%%%%%%%%%%%%%%%%%%%%%%%%%%%%%%%%%%%%%%%%%%%%%%%%%%%%%%%%%%

\section{Linear forms on vector bundles}\label{sec:linear}

In order to define linear forms, we recall a few facts about tangent
and cotangent bundles of vector bundles.

%%%%%%%%%%%%%%%%%%%%%%%%%%%%%%%%%%%%%%%%%%%%%%%%%%%%%%%%%%%%%%%%%%
\subsection{Tangent and cotangent bundles of vector
bundles}\label{subsec:tancot}

Let $q_A: A \rmap M$ be a vector bundle, and let $TA$ be the tangent
bundle of the total space $A$. Besides its natural vector bundle
structure over $A$, with projection map denoted by $p_A:TA\rmap A$,
it is also a vector bundle over $TM$, with respect to the map $Tq_A:
TA \rmap TM$.

It is useful to consider a coordinate description of these bundles.
Let $(x^j)$ be coordinates on $M$, $j=1,\ldots,\mbox{dim}(M)$, and
let $\{e_d\}$ be a basis of local sections of $A$,
$d=1,\ldots,\mbox{rank}(A)$. The corresponding coordinates on $A$
are denoted by $(x^j,u^d)$, and tangent coordinates on $TA$ by
$(x^j,u^d,\dot{x}^j,\dot{u}^d)$. In this notation, given $x=(x^j)$,
the coordinates $(u^d)$ specify a point in $A_x$, $(\dot{x}^j)$ a
point in $T_xM$, whereas $(\dot{u}^d)$ determines a point on a
second copy of $A_x$, tangent to the fibres of $A\rmap M$. Note that
$p_A(x^j,u^d,\dot{x}^j,\dot{u}^d)=(x^j,u^d)$, and
$Tq_A(x^j,u^d,\dot{x}^j,\dot{u}^d)=(x^j,\dot{x}^j)$.

Similarly, consider the cotangent bundle $T^*A$, with local
coordinates $(x^j,u^d,p_j,\xi_d)$, where $(p_j)$ determines a point
in $T^*_xM$, and $(\xi_d)$ a point in $A_x^*$, dual to the direction
tangent to the fibres of $A\rmap M$. In this case, besides the
natural vector bundle structure $c_A: T^*A\rmap A$,
$c_A(x^j,u^d,p_j,\xi_d)=(x^j,u^d)$, $T^*A$ is also a vector bundle
over $A^*$ \cite{Mac-Xu}, with respect to the projection map given
in coordinates by
\begin{equation}\label{eq:r}
r: T^*A\rmap A^*,\;\; r(x^j,u^d,p_j,\xi_d) = (x^j,\xi_d).
\end{equation}

The total spaces $TA$ and $T^*A$ are examples of \textit{double
vector bundles}, see \cite{Mac-book,Prad}. They fit into the
following commutative diagrams:
$$
\xymatrix{ TA  \ar[r]^{T{q_A}} \ar[d]_{p_A} & TM \ar[d]^{p_M} \\
A \ar[r]_{q_A} &  M } \qquad \;\;\;\;
\xymatrix{ T^*A  \ar[r]^{r} \ar[d]_{c_A} & A^* \ar[d]^{q_{A^*}} \\
A \ar[r]_{q_A} &  M }
$$
where
\begin{equation}
p_M:TM\rmap M, \; p_M(x^j,\dot{x}^j)=(x^j), \;\;\; q_{A^*}:A^*\rmap
M, \; q_{A^*}(x^j,\xi_d)=(x^j),
\end{equation}
are the natural projections. Recall, see e.g. \cite{Mac-book}, that
the intersection of the kernels of the top and left arrows on each
diagram defines a vector bundle over $M$, known as the
\textit{core}. In the case of $TA$, the core is identified with
$A\rmap M$, with coordinates $(x^j,\dot{u}^d)$; for $T^*A$, the core
is $T^*M$, with coordinates $(x^j,p_j)$.

%%%%%%%%%%%%%%%%%%%%%%%%%%%%%%%%%%%%%%%%%%%%%%%%%%%%%%%%%%%%%%%%%%%%%
\subsection{The structure of linear forms on vector
bundles}\label{subsec:struc} Let $A\rmap M$ be a vector bundle, with
local coordinates $(x^j,u^d)$, and let us consider the $k$-fold
direct sum of $TA$ over $A$,
$$
\oplus^k_A TA := TA \times_A \ldots \times_A TA,
$$
locally described by coordinates $(x^j,u^d,
\dot{x}_1^j,\ldots,\dot{x}_k^j,\dot{u}_1^d,\ldots,\dot{u}_k^d)$. It
is a vector bundle over $A$, with projection map
$$
(x^j,u^d,
\dot{x}_1^j,\ldots,\dot{x}_k^j,\dot{u}_1^d,\ldots,\dot{u}_k^d)
\mapsto (x^j,u^d),
$$
and also a vector bundle over $\oplus^k TM = TM\times_M \ldots
\times_M TM$, with projection map
$$
(x^j,u^d,
\dot{x}_1^j,\ldots,\dot{x}_k^j,\dot{u}_1^d,\ldots,\dot{u}_k^d)\mapsto
(x^j,\dot{x}_1^j,\ldots,\dot{x}_k^j).
$$

Given a $k$-form $\Lambda \in \Omega^k(A)$ on the total space of
$A\rmap M$, let us consider the induced maps
\begin{align}
& \Lambda^\sharp: \oplus_A^{k-1} TA \rmap T^*A, \;\;\;
\Lambda^\sharp(U_1,\ldots,U_{k-1})=i_{U_{k-1}}\ldots
i_{U_1}\Lambda, \label{eq:lsharp}\\
& \overline{\Lambda}: \oplus_A^k TA \rmap \mathbb{R},\qquad
\;\;\;\;\; \overline{\Lambda}(U_1,\ldots,U_k)= i_{U_{k}}\ldots
i_{U_1}\Lambda,\label{eq:lbar}
\end{align}
which are alternating and linear in each of their
entries\footnote{Notice that, since $\Lambda^\sharp$ (resp.
$\overline{\Lambda}$) is \emph{multilinear} in its entries, it is
\textit{not} a vector-bundle morphism from the direct sum
$\oplus^{k-1} TA \to A$ (resp. $\oplus^k TA\to A$) to $T^*A\to A^*$,
unless $k=2$ (resp. $k=1$).}.

\begin{definition}\label{def:linear}
A $k$-form $\Lambda$ is called \textbf{linear} if the induced map
$\Lambda^\sharp$ \eqref{eq:lsharp} is a morphism of vector bundles
with respect to the vector bundle structures $\oplus_A^{k-1} TA
\rmap \oplus^{k-1}TM$ and $T^*A\rmap A^*$. The space of linear
$k$-forms on $A$ is denoted by $\Omega_{\mathrm{lin}}^k(A)$.
\end{definition}
In particular, $\Lambda^\sharp$ covers a base map
$\lambda:\oplus^{k-1}TM \rmap A^*$,
\begin{equation}\label{eq:Lambda}
\xymatrix{\oplus_A^{k-1} TA  \ar[r]^{\;\;\;\;\;\Lambda^\sharp} \ar[d] & T^*A \ar[d]_r \\
 \oplus^{k-1}TM \ar[r]_{\;\;\;\;\; \lambda} &  A^* .}
\end{equation}
The map $\lambda$ is skew symmetric on its entries, so it can be
viewed as a vector-bundle map $\wedge^{k-1}TM \rmap A^*$. Its
transpose is the vector-bundle map
\begin{equation}
\lambda^t: A \rmap \wedge^{k-1}T^*M.
\end{equation}

A simple computation in coordinates shows the following.

\begin{lemma}\label{lem:linear}
Given a $k$-form $\Lambda\in \Omega^k(A)$, the following are
equivalent:
\begin{enumerate}
\item $\Lambda$ is linear.
\item In local coordinates $(x^j,u^d)$ on $A$, $\Lambda$ has the form
\begin{align}\label{eq:linearlocal}
\Lambda = &\frac{1}{k!}\Lambda_{i_1\ldots i_k,d}(x)u^d \dd
x^{i_1}\wedge\ldots\wedge \dd x^{i_k} + \\\nonumber &
\frac{1}{(k-1)!}\lambda_{i_1\ldots i_{k-1}d}(x)\dd x^{i_1}\wedge
\ldots \wedge \dd x^{i_{k-1}}\wedge \dd u^d,
\end{align}
where $\lambda_{i_1\ldots i_{k-1}d} =\langle \lambda(
{\partial_{x^{i_1}}}, \ldots, {\partial_{x^{i_{k-1}}}}),e_d\rangle$,
and $\partial_{x^j}=\frac{\partial}{\partial x^{j}}$.

\item The map $\overline{\Lambda}: \oplus^k_A TA \rmap \mathbb{R}$ defines a
vector-bundle map
\begin{equation}\label{eq:Lvbmap}
\xymatrix{\oplus_A^{k} TA  \ar[r]^{\;\;\;\;\;\overline{\Lambda}} \ar[d] & \mathbb{R} \ar[d] \\
 \oplus^{k}TM \ar[r]_{} &  \{*\} . }
\end{equation}
\end{enumerate}

\end{lemma}

Given a vector-bundle map $\mu: A\rmap \wedge^k T^*M$, let us
consider the linear $k$-form $\Lambda_\mu$ on $A$ given at a point
$u\in A$ by
\begin{equation}\label{eq:lambamu}
(\Lambda_\mu)_u := Tq_A|_u^t \mu (u).
\end{equation}
 In local coordinates $(x^i,u^d)$ on $A$, $\Lambda_\mu$ is
written as
\begin{equation}\label{eq:lambdamulocal}
(\Lambda_\mu)_{u} = \frac{1}{k!}\mu_{i_1\ldots i_k,d}(x)u^d \dd
x^{i_1}\wedge \ldots \wedge \dd x^{i_k},
\end{equation}
where $\mu_{i_1\ldots i_k,d}$ is defined by
$$
\mu_{i_1\ldots i_k,d} = \SP{\mu(e_d),\frac{\partial}{\partial
x^{i_1}}\wedge \ldots \wedge \frac{\partial}{\partial x^{i_k}}}.
$$

\begin{example}\label{ex:can1}
When $k=1$, a direct computation in coordinates shows that the
linear 1-form $\Lambda_\mu$, defined by the vector-bundle map $\mu:
A\rmap T^*M$, satisfies
\begin{equation}\label{eq:taut}
\Lambda_\mu = \mu^* \theta_{can},
\end{equation}
where $\theta_{can}=p_i\dd x^i$ is the canonical 1-form on $T^*M$.
(When $A=M \rmap M$ is the vector bundle with zero fibres,
\eqref{eq:taut} recovers the well-known ``tautological'' property
$\mu^*\theta_{can}=\mu$.)
\end{example}

\begin{lemma}\label{lem:zeromap}
A linear $k$-form $\Lambda$ covers the fibrewise zero map in
\eqref{eq:Lambda} if and only if it is of the form $\Lambda_\mu$ (as
in \eqref{eq:lambamu}) for a vector bundle map $\mu:A\rmap \wedge^k
T^*M$.
\end{lemma}

\begin{proof}
We can use Lemma~\ref{lem:linear}, or argue more globally as
follows. Consider the projection $r: T^*A \rmap A^*$, as in
\eqref{eq:r}. One can directly check that
$$
\ker(r)_u = (\ker(Tq_A|_u))^\circ = \mathrm{im}(T q_A |_u^t),
$$
where $^\circ$ stands for the annihilator. It follows from
\eqref{eq:lambamu} that $r\circ \Lambda_{\mu}^\sharp = 0$, which
means that $\Lambda_\mu$ covers the fibrewise zero map in
\eqref{eq:Lambda}. Conversely, if $\Lambda$ covers the fibrewise
zero map, then $r\circ \Lambda^\sharp = 0$; so, given
$U_1,\ldots,U_{k}\in T_uA$, $\Lambda^\sharp(U_1,\ldots,U_{k-1}) =
Tq_A|_u^* \alpha$ for some $\alpha\in T^*_{q_A(u)}M$. Since
$\Lambda$ is skew symmetric, we conclude that
$\Lambda(U_1,\ldots,U_k)$ only depends on $Tq_A|_u(U_j)$,
$j=1,\ldots,k$. Hence, for each $u\in A$, there exists $\mu(u) \in
\wedge^l T^*_{q_A(u)}M$ such that $(\Lambda)_u = Tq_A|_u^* \mu (u)$.
The linear dependence of $\mu(u)$ on $u$ follows from the linear
dependence of $(\Lambda)_u$ on $u$, see \eqref{eq:linearlocal}; the
resulting vector bundle map $\mu: A\rmap \wedge^k T^*M$ is smooth by
the local expression \eqref{eq:lambdamulocal}.
\end{proof}

\begin{proposition}\label{prop:struc}
There is a one-to-one correspondence between linear $k$-forms
$\Lambda$ on $A$, covering a map $\lambda: \oplus^{k-1}TM\rmap A^*$,
and pairs $(\mu, \nu)$, where $\mu:A \rmap \wedge^{k-1}T^*M$ and
$\nu:A \rmap \wedge^{k}T^*M$ are vector bundle morphisms. The
correspondence is given by
\begin{equation}\label{eq:corresp}
\Lambda = \dd\Lambda_\mu + \Lambda_\nu,
\end{equation}
where $\mu = (-1)^{k-1}\lambda^t$.
\end{proposition}

\begin{proof}
Let $\Lambda$ be a linear $k$-form on $A$, and set
$\mu=(-1)^{k-1}\lambda^t$. A direct computation using the local
expression \eqref{eq:lambdamulocal} and Lemma~\ref{lem:linear} shows
that the $k$-form $d\Lambda_\mu$ is linear and covers the same map
$\lambda$, hence the linear $k$-form $\Lambda-\dd\Lambda_\mu$ covers
the fibrewise zero map. By Lemma \ref{lem:zeromap}, there is a
unique $\nu: A \rmap \wedge^k T^* M$ such that $\Lambda -
\dd\Lambda_\mu = \Lambda_\nu$.
\end{proof}

A direct consequence of \eqref{eq:corresp} is that if $\Lambda$ is a
linear form, then so is $\dd\Lambda$.
\begin{example}
Let $\Lambda\in \Omega^2(A)$ be a linear 2-form with $\dd\Lambda=0$.
According to the previous proposition, we can write it as
$\Lambda=\dd\Lambda_\mu + \Lambda_\nu$, and $\Lambda$ being closed
amounts to $\dd \Lambda_\nu =0$; this condition immediately implies
that $\nu=0$, so $\Lambda = d\Lambda_\mu$. Using \eqref{eq:taut}, it
follows that
$$
\Lambda = (\lambda^t)^* \omega_{can},
$$
where $\omega_{can}=-\dd\theta_{can} = \mathrm{d}x^i\wedge
\mathrm{d}p_i$ the canonical symplectic form on $T^*M$ (see
\cite[Sec.~7.3]{KU}, and also \cite[Prop.~4.3]{BCO}).
\end{example}

%%%%%%%%%%%%%%%%%%%%%%%%%%%%%%%%%%%%%%%%%%%%%%%%%%%%%%%%%%%%%%%%%%%%%%%%%%%%%%%%%%%%%%%%%%%%
\subsection{Tangent lifts}\label{subsec:tlift}

We now briefly discuss linear forms obtained via the \textit{tangent
lift} operation \cite{GU,YI} (see also \cite{BCO} and
\cite{Mac-Xu}), that assigns to any $k$-form on a manifold $M$ a
linear $k$-form on the total space of its tangent bundle $p_M:TM\to
M$.

Let us consider the operation
\begin{equation}\label{eq:tau}
\tau: \Omega^l(M)\rmap \Omega^{l-1}(TM),\;\;\; \tau(\beta)|_X :=
(Tp_M|_X)^t(i_X\beta),
\end{equation}
where $X\in TM$ and $l\geq 1$; i.e., for $U_1,\dots,U_{l-1}\in
T_X(TM)$,
$$
i_{U_{l-1}}\ldots i_{U_1}\tau(\beta)|_X
=\beta(X,Tp_M(U_1),\ldots,Tp_M(U_{l-1})).
$$
In the notation of Section \ref{subsec:struc}, $\tau(\beta)$ is a
linear $(l-1)$-form on the vector bundle $A=TM$ of type $\Lambda_\nu$,
where
$$
\nu: TM\rmap \wedge^{l-1}T^*M,\;\;\; \nu(X)=i_X\beta.
$$

It directly follows from Example~\ref{ex:can1} that, if $\omega \in
\Omega^2(M)$, then $\tau(\omega)=(\omega^\sharp)^*\theta_{can}$,
where $\theta_{can}=p_i \dd x^i$ is the canonical 1-form on $T^*M$.

The \textbf{tangent lift} operation,
\begin{equation}\label{eq:tglift}
\Omega^k(M)\rmap \Omega^k(TM),\;\; \alpha \mapsto \alpha_T,
\end{equation}
assigns to $\alpha \in \Omega^k(M)$ the form $\alpha_T \in
\Omega^k(TM)$ defined by the Cartan-like formula
\begin{equation}\label{eq:cartanlike}
\alpha_T = \dd \tau(\alpha) + \tau(\dd \alpha).
\end{equation}
It follows directly from \eqref{eq:cartanlike} that $\alpha_T$ is
linear and that the operation \eqref{eq:tglift} is compatible with
exterior derivatives, in the sense that $(\dd \alpha)_T = \dd
\alpha_T$.

We will also need an equivalent characterization of the tangent
lift, see e.g. \cite{GU}. Given $\alpha \in \Omega^k(M)$, consider
the associated map
$$
\overline{\alpha}: \oplus^k TM \rmap \mathbb{R},\;\;
(X_1,\dots,X_k)\mapsto \alpha(X_1,\ldots,X_k).
$$
Let $\prod^k_{Tp_M}T(TM)$ denote the fibred product with respect to
the vector bundle
$$
Tp_M: T(TM)\rmap TM, \;\; Tp_M(x^j,\dot{x}^j,\delta x^j, \delta
\dot{x}^j) = (x^j,\delta x^j),
$$
where $(x^j,\dot{x}^j,\delta x^j, \delta \dot{x}^j)$ are the local
coordinates on $T(TM)$ induced by the tangent coordinates
$(x^j,\dot{x}^j)$ on $TM$. We have a natural identification
$T(\oplus^kTM)\cong \prod^k_{Tp_M}T(TM)$, so we can view the
differential of the function $\overline{\alpha}$ in $C^\infty(\oplus^kTM)$ as a map
$$
\dd \overline{\alpha} : \prod^k_{Tp_M}T(TM) \rmap \mathbb{R}.
$$
Note that the canonical involution
\begin{equation}\label{eq:J}
J_M : T(TM)\rmap T(TM), \;\;\; J_M(x^j,\dot{x}^j,\delta x^j, \delta
\dot{x}^j) = (x^j,\delta x^j, \dot{x}^j, \delta \dot{x}^j),
\end{equation}
induces an identification
$$
J_M^{(k)}: \prod^k_{p_{TM}} T(TM)\rmap \prod^k_{Tp_M}T(TM).
$$
One can prove (see e.g. \cite{GU}) that, given $\alpha \in
\Omega^k(M)$, its tangent lift $\alpha_T \in \Omega^k(TM)$ is
uniquely determined by the condition
\begin{equation} \label{eq:alphaT}
\overline{\alpha_T} = \dd \overline{\alpha}\circ J_M^{(k)}:
\prod^k_{p_{TM}} T(TM) \rmap \mathbb{R}.
\end{equation}

%%%%%%%%%%%%%%%%%%%%%%%%%%%%%%%%%%%%%%%%%%%%%%%%%%%%%%%%%%%%%%%%%%%%%%%%%%%%%%%%%%%%%%%%%%%%%
\section{Linear forms on Lie algebroids}\label{sec:alg}

%%%%%%%%%%%%%%%%%%%%%%%%%%%%%%%%%%%%%%%%%%%%%%%%%%%%%%%%%%%%%%%%%%%%%
\subsection{Core and linear sections}\label{subsec:corelinear}

Any vector bundle that fits into a double vector bundle admits two
distinguished types of sections, known as \textit{linear} and
\textit{core} sections; a detailed discussion can be found e.g. in
\cite[Sec. 2.3]{GM} and \cite{Mac-book}. For our purposes, we are
mostly interested in the particular (double) vector bundles
$\oplus^k_A TA$ (and, later in Section~\ref{sec:mult}, also in
$\oplus^k_A T^*A$), so we restrict ourselves to these cases.

Each section $u$ of a vector bundle $A\to M$ defines a core section
$\widehat{u}$ and a linear section $Tu$ of $TA\to TM$ as follows.
The tangent bundle of $A$ along its zero section $M\hookrightarrow
A$ naturally splits as $TA|_M = TM \oplus A$, and we can define, for
each $X\in T_xM$, $\widehat{u}(X):= X(x) + u(x)$, where the sum is
with respect to the previous decomposition of $TA|_M$. The linear
section $Tu:TM\to TA$ is obtained by applying the tangent functor to
$u: M\to A$.

Let us consider local coordinates $(x^j)$ on $M$, a basis of local
sections $\{e_d\}$ of $A$, and dual basis $\{e^d\}$ of $A^*$. As in
Section \ref{sec:linear}, we denote the corresponding coordinates on
$A$ by $(x^j,u^d)$, and on $A^*$ by $(x^j,\xi_d)$, while coordinates
on $TA$ are denoted by $(x^j,u^d,\dot{x}^j,\dot{u}^d)$, and on
$T^*A$ by $(x^j,u^d,p_j,\xi_d)$.

The core and linear sections of $TA\to TM$ defined by a local
section $e_a$ of $A$ are explicitly given by
\begin{equation}\label{eq:secTA}
\he_a(x^j,\dot{x}^j)=(x^j,0,\dot{x}^j, \delta_a^d),\;\;\;\;
Te_a(x^j,\dot{x}^j)=(x^j,\delta_a^d,\dot{x}^j,0),
\end{equation}
where $\delta_a^d$ is the $d$-th component of $e_a$, i.e., $1$ if
$d=a$ or zero otherwise.

More generally, $e_a$ locally defines two types of local sections of
$\oplus_A^k TA \rmap \oplus^k TM$ as follows: the first type is
given, for each $n\in \{1,\dots,k\}$, by
\begin{equation}\label{eq:secPiTA1}
\he_{a,n}(\dot{x}_1\oplus \ldots \oplus \dot{x}_k): =
\widehat{0}(\dot{x}_1)\oplus \ldots \oplus
\widehat{0}(\dot{x}_{n-1})\oplus \he_a(\dot{x}_n) \oplus
\widehat{0}(\dot{x}_{n+1})\oplus \ldots \oplus
\widehat{0}(\dot{x}_k),
\end{equation}
where $\dot{x}_l=(x^j,\dot{x}_l^j)$ belongs to the $l$-th component
of $\oplus^k TM$ and $\widehat{0}
(\dot{x}_l)=(x^j,0,\dot{x}^j_l,0)$; the second type is
\begin{equation}\label{eq:secPiTA2}
(Te_a)^k(\dot{x}_1\oplus \ldots \oplus \dot{x}_k):=
Te_a(\dot{x}_1)\oplus \ldots \oplus Te_a(\dot{x}_k).
\end{equation}

The sections $\he_{a,n}$ and $(Te_a)^k$ are the {core} and {linear}
sections on $\oplus_A^k TA$, respectively.
%, of double
%vector bundles (see e.g. \cite[Sec. 2.3]{GM} and \cite{Mac-book} for
%details, including a coordinate-free viewpoint).
A key property is that they generate the module of local sections of
$\oplus_A^k TA \rmap \oplus^kTM$. Note also that, under the natural
projection $\oplus_A^kTA \rmap A$, core sections $\he_{a,n}$ are
sent to the zero section of $A\rmap M$, while linear sections
$(Te_a)^k$ map to the section $e_a$.

%%%%%%%%%%%%%%%%%%%%%%%%%%%%%%%%%%%%%%%%%%%%%%%%%%%%%%%%%%%%%%%%%%%%
\subsection{Tangent Lie
algebroids}\label{subsec:Liealg}

Suppose that $A\rmap M$ carries a Lie algebroid structure (see e.g.
\cite{CW,Mac-book}), with Lie bracket $[\cdot,\cdot]_{\st{A}}$ on
$\Gamma(A)$ and anchor map $\rho_{\st{A}}:A \rmap TM$. Then the
vector bundle $TA\rmap TM$ inherits a natural Lie algebroid
structure, known as the \textit{tangent Lie algebroid}, see e.g.
\cite{Mac-Xu1}. We will need local expressions for the tangent Lie
algebroid
 in terms of the coordinates introduced in Section
\ref{subsec:corelinear}.

The Lie algebroid $A\rmap M$ is locally determined by structure
functions $\rho_a^j$ and $C_{ab}^c$ defined by
\begin{equation}\label{eq:liestruc}
\rho_{\st{A}}(e_a)=\rho^j_a\frac{\partial}{\partial x^j},\;\;\;
[e_a,e_b]_{\st{A}}=C_{ab}^c e_c.
\end{equation}
The tangent Lie algebroid structure on $TA \rmap TM$ is defined in
terms of core and linear sections \eqref{eq:secTA} by
\begin{align}
&[\widehat{e}_a,\widehat{e}_b]_{\st{TA}}=0, \;\;
[T{e}_a,\widehat{e}_b]_{\st{TA}}=C^{c}_{ab}\widehat{e}_c, \;\;
[Te_a,Te_b]_{\st{TA}}= C_{ab}^cTe_c + \dot{x}^i \frac{\partial{C}^{c}_{ab}}{\partial x^i}\widehat{e}_c,\label{eq:TA1} \\
& \rho_{\st{TA}}(Te_a)=\rho^{j}_{a}\frac{\partial}{\partial x^j} +
\dot{x}^i \frac{\partial \rho_a^j}{\partial x^i} \frac{\partial}{\partial \dot{x}^j},\;\;\;
\rho_{\st{TA}}(\widehat{e}_a)=\rho^{j}_{a}\frac{\partial}{\partial
\dot{x}^j}.\label{eq:TA2}
\end{align}
In \eqref{eq:TA2}, we have identified points in $T(TM)$, written in
coordinates as $(x^j,\dot{x}^j,\delta x^j, \delta \dot{x}^j)$, with
tangent vectors
$$
\delta x^j \frac{\partial}{\partial x^j} + \delta \dot{x}^j
\frac{\partial}{\partial \dot{x}^j} \Big |_{(x^j,\dot{x}^j)}.
$$

We notice that the tangent Lie algebroid induces a Lie algebroid
structure on the direct sum $\oplus^k_ATA\rmap \oplus^kTM$. This is
a general property of \textit{$\mathcal{VB}$-algebroids} \cite[Sec.
2.1]{GM}, which we directly verify in this example. A simple
consequence of \eqref{eq:TA1} and \eqref{eq:TA2} is that if $U$ and
$V$ are local sections of $TA\rmap TM$, each of type $\he_a$ or
$Te_a$, then
\begin{equation}\label{eq:dsum}
Tp_M (\rho_{\st{TA}}(U)) = \rho_{\st{A}}(p_A(U)), \;\;\;
p_A([U,V]_{\st{TA}})=[p_A(U),p_A(V)]_{\st{A}},
\end{equation}
where $p_M:TM\rmap M$ and $p_A: TA\rmap A$ are the natural
projections. It follows from the first equation in \eqref{eq:dsum}
that if $U_1\oplus \ldots \oplus U_k \in \oplus_A^k TA$ is of type
\eqref{eq:secPiTA1} or \eqref{eq:secPiTA2}, then
$$
Tp_M(\rho_{\st{TA}}(U_l))=Tp_M(\rho_{\st{TA}}(U_m)), \;\; \forall
l,m \in \{1,\ldots,k\}.
$$
As a result, $(\rho_{\st{TA}}(U_1),\ldots, \rho_{\st{TA}}(U_k))$
defines an element in $\prod^k_{Tp_M}T(TM)$. Using the natural
identification $\prod^k_{Tp_M}T(TM) = T(\oplus^k TM)$, we obtain a
vector bundle map $\rho_{k}: \oplus^k_A TA \rmap T(\oplus^kTM)$,
\begin{equation}\label{eq:anchotoplus}
\rho_{k}(U_1\oplus\ldots\oplus U_k):=\rho_{\st{TA}}(U_1)\oplus
\ldots \oplus \rho_{\st{TA}}(U_k).
\end{equation}
Writing $\oplus^kTM$ in local coordinates
$(x^j,\dot{x}_1^j,\ldots,\dot{x}_k^j)$, we have the following
explicit formulas:
\begin{align}
& \rho_k(\he_{a,n}) =
\rho_a^j\frac{\partial}{\partial \dot{x}_n^j},\label{eq:rhocore}\\
& \rho_k((Te_a)^k) =\rho_a^j \frac{\partial}{\partial x^j} +
\sum_{n=1}^{k}W^j_{a,n}\frac{\partial}{\partial
\dot{x}^j_n},\label{eq:rholinear}
\end{align}
where $W^j_{a,n} = \dot{x}_n^i \frac{\partial \rho_a^j}{\partial x^i}
\in C^\infty(\oplus^k TM)$.

The second equation in \eqref{eq:dsum} implies that if $U_1\oplus
\ldots \oplus U_k$ and $V_1\oplus \ldots \oplus V_k$ are local
sections of $\oplus^k_A TA \rmap \oplus^k TM$ of type
\eqref{eq:secPiTA1} or \eqref{eq:secPiTA2}, then
\begin{equation}\label{eq:oplusbr}
[U_1\oplus \ldots \oplus U_k, V_1\oplus \ldots \oplus V_k]_{k} :=
[U_1,V_1]_{\st{TA}}\oplus \ldots \oplus [U_k,V_k]_{\st{TA}}
\end{equation}
is a well-defined local section of $\oplus^k_A TA\rmap \oplus^k TM$.
Explicitly, we have:
\begin{align}
& [\he_{a,n},\he_{b,m}]_k= 0,\label{eq:Tbr1}\\
& [(Te_a)^k,\he_{b,m}]_k= C_{ab}^d \he_{d,m}\label{eq:Tbr2}\\
&[(Te_a)^k,(Te_b)^k]_k= C_{ab}^d (Te_d)^k + \sum_{n=1}^k \dot{x}_{n}^i \frac{\partial{C}^{d}_{ab}}{\partial x^i}
\he_{d,n}.\label{eq:Tbr3}
\end{align}

The induced Lie algebroid structure on $\oplus^k_A TA\rmap \oplus^k
TM$ is defined by $\rho_{k}$ and the extension of
$[\cdot,\cdot]_{k}$ to all sections via the Leibniz rule\footnote{We
adopt the simplified notation $\rho_k$, $[\cdot,\cdot]_k$, instead
of $\rho_{\st{\oplus^k_A TA}}$ and $[\cdot,\cdot]_{\st{\oplus^k_A
TA}}$; in particular, $\rho_1=\rho_{\st{TA}}$ and $[\cdot,\cdot]_1 =
[\cdot,\cdot]_{\st{TA}}$.}.

%%%%%%%%%%%%%%%%%%%%%%%%%%%%%%%%%%%%%%%%%%%%%%%%%%%%%%%%%%%%%%%%%%%%%%%%%%%%%%%%%%%%%%%%%%%
\subsection{IM-forms}

Let $\Lambda \in \Omega^k(A)$ be a linear $k$-form on a Lie
algebroid $A\rmap M$, $k\geq 1$. Following Prop.~\ref{prop:struc},
let $\mu: A\rmap \wedge^{k-1}T^*M$ and $\nu: A\rmap \wedge^k T^*M$
be the vector-bundle maps such that $\Lambda=\dd\Lambda_\mu +
\Lambda_\nu$. Let us consider the bundle map
\begin{equation}\label{eq:Lvbmap2}
\xymatrix{\oplus_A^{k} TA  \ar[r]^{\;\;\;\;\;\overline{\Lambda}} \ar[d] & \mathbb{R} \ar[d] \\
 \oplus^{k}TM \ar[r]_{} &  \{*\}. }
\end{equation}

The following is the main result of this section.

\begin{theorem}\label{thm:mainresult}
The map \eqref{eq:Lvbmap2} is a Lie algebroid morphism if and only
if the following holds for all $u,v \in \Gamma(A)$:
\begin{align}
i_{\rho(u)}\mu(v)& = - i_{\rho(v)}\mu(u) \label{eq:IM1}\\
 \mu([u,v])& = \Lie_{\rho(u)}\mu(v)-i_{\rho(v)}\dd\mu(u) -
i_{\rho(v)}\nu(u)\label{eq:IM2}\\
 \nu([u,v])&= \Lie_{\rho(u)}\nu(v) - i_{\rho(v)}\dd\nu(u).
\label{eq:IM3}
\end{align}
\end{theorem}

For a Lie algebroid $A\to M$ and vector-bundle maps
$$
\mu: A\rmap\wedge^{k-1}T^*M,\;\;\; \nu: A\rmap \wedge^k T^*M,
\;\;k\geq 1,
$$
we say that the pair $(\mu,\nu)$ is an \textbf{IM $k$-form} on $A$
if conditions \eqref{eq:IM1}, \eqref{eq:IM2} and \eqref{eq:IM3} are
satisfied. The terminology IM stands for \textit{infinitesimally
multiplicative}, and it will be clarified in Section \ref{sec:mult}.
The space of IM $k$-forms on $A$ is denoted by
$\Omega_{\mathrm{IM}}(A)$.

We note that Theorem~\ref{thm:mainresult} can be alternatively
phrased in terms of the map $\Lambda^\sharp$ \eqref{eq:Lambda}, as
this map is a Lie algebroid morphism if and only if so is
$\overline{\Lambda}$.

\begin{remark} Given an IM-form $(\mu,\nu)$, it follows from \eqref{eq:IM2}, using the skew-symmetry and Jacobi
identity for the Lie algebroid bracket $[\cdot,\cdot]$ on
$\Gamma(A)$, that $\nu$ automatically satisfies
\begin{align}
& i_{\rho(u)}\nu(v)=-i_{\rho(v)}\nu(u),\label{eq:nuextra1}\\
&  i_{\rho(w)}(\Lie_{\rho(v)}\nu(u)- \Lie_{\rho(u)}\nu(v)) + c.p. =
0 \label{eq:nuextra2}
\end{align}
for all $u,v, w \in \Gamma(A)$, where $c.p.$ stands for
\textit{cyclic permutations} in $u, v, w$.
\end{remark}

\begin{example}\label{ex:IM1}
Consider a Lie algebroid $A\to M$ and a $k$-form $\eta
\in\Omega^{k}(M)$. Then the pair $(\mu,\nu)$ of vector-bundle maps
$$
\mu:A\to \wedge^{k-1}T^*M, \;\; \mu(u)=-i_{\rho(u)}\eta,\;\;\mbox{
and }\;\; \nu: A\to \wedge^kT^*M, \;\; \nu(u)=-i_{\rho(u)}\dd\eta,
$$
defines an IM $k$-form on $A$.
\end{example}

\begin{example}\label{ex:IM2}
Let $A\to M$ be a Lie algebroid, and let $\phi\in \Omega^{k+1}(M)$
be such that $i_{\rho(u)}\dd\phi=0$, $\forall u\in \Gamma(A)$. One
directly checks that the vector-bundle map $\nu: A\rmap
\wedge^{k}T^*M$ given by
\begin{equation}\label{eq:rel}
 \nu(u):= -i_{\rho(u)}\phi
\end{equation}
verifies \eqref{eq:IM3}. The particular IM $k$-forms $(\mu,\nu)$ on
$A$ for which $\nu$ is given as in \eqref{eq:rel} for a closed form
$\phi \in \Omega^{k+1}(M)$ are called \textbf{IM $k$-forms relative
to $\phi$}. These special types of IM forms have first appeared in
\cite{bcwz} (for $k=2$), and more recently in \cite{AC} (for
arbitrary $k$), in the study of multiplicative forms (see
Section~\ref{sec:mult}).
\end{example}

\begin{remark}
Let $\iota_\mathcal{C}:\mathcal{C}\hookrightarrow M$ be an orbit of
the Lie algebroid $A\to M$, i.e., an integral leaf of the
distribution $\rho(A)\subset TM$. If $(\mu,\nu)$ is an IM $k$-form
on $A$, then we have induced forms $\mu_\mathcal{C}\in
\Omega^k(\mathcal{C})$ and $\nu_\mathcal{C}\in
\Omega^{k+1}(\mathcal{C})$ defined by
$$
i_{\rho(u)}\mu_\mathcal{C}=\iota_\mathcal{C}^*\mu(u),\;\;\;
i_{\rho(u)}\nu_\mathcal{C}=\iota_\mathcal{C}^*\nu(u).
$$
It follows from \eqref{eq:IM1} and \eqref{eq:nuextra1} that the
formulas above do define differential forms on $\mathcal{C}$;
moreover, \eqref{eq:IM2} implies that
$\dd\mu_\mathcal{C}=\nu_{\mathcal{C}}$. In particular, we see that
any IM $k$-form on a transitive Lie algebroid is like the one in
Example~\ref{ex:IM1}.
\end{remark}

In order to prove Thm.~\ref{thm:mainresult}, we need some lemmas. We
work in local coordinates $(x^j,u^d)$ on $A$, induced by coordinates
$(x^j)$ on $M$ and the choice of a basis of local sections $\{e_d\}$
of $A$ (see Section \ref{subsec:corelinear}).

\begin{lemma}
Let $\dot{x}=(\dot{x}_1, \ldots ,\dot{x}_k) \in \oplus^kTM$, where
$\dot{x}_l=(x^j,\dot{x}_l^j)$ belongs to the $l$-th copy of $TM$.
Then:
\begin{align}
& \overline{\Lambda}(\he_{a,n}(\dot{x}_1,\ldots,\dot{x}_k))= (-1)^{n-1}I^{\dot{x}}_{k,n+1}I^{\dot{x}}_{n-1,1}\mu(e_a), \label{eq:lambdacore}\\
& \overline{\Lambda}((Te_a)^k(\dot{x}_1,\ldots,\dot{x}_k)) =
I^{\dot{x}}_{k,1}(\dd\mu(e_a)+\nu(e_a)),\label{eq:lambdalinear}
\end{align}
seen as functions in $C^\infty(\oplus^kTM)$ (see \eqref{eq:Iop} for
notation).
\end{lemma}

\begin{proof}
Writing $\Lambda=\dd\Lambda_\mu + \Lambda_\nu$ and recalling the
local expressions of $\Lambda_\mu$ and $\Lambda_\nu$ (see
\eqref{eq:lambdamulocal}), we have
\begin{align}
\Lambda |_{(x^j,u^d)} = & \frac{1}{(k-1)!}u^d \dd\mu_{i_1\ldots i_{k-1},d}(x)\wedge \dd x^{i_1}\wedge \ldots \wedge \dd x^{i_{k-1}} +\\
& \frac{1}{(k-1)!}\mu_{i_1\ldots i_{k-1},d}(x)\dd u^d \wedge
\dd x^{i_1}\wedge \ldots \wedge \dd x^{i_{k-1}} +\nonumber \\
& \frac{1}{k}\nu_{i_1\ldots i_k,d}(x)u^d \dd x^{i_1}\wedge \ldots
\wedge \dd x^{i_{k}}. \nonumber
\end{align}
We write points in $TA$ with coordinates
$(x^j,u^d,\dot{x}^j,\dot{u}^d)$ in terms of horizontal tangent
vectors $\frac{\partial}{\partial x^j}$ and vertical tangent vectors
$\frac{\partial}{\partial u^d}$ as
$$
\dot{x}^j\frac{\partial}{\partial x^j} +
\dot{u}^d\frac{\partial}{\partial u^d} \Big|_{(x^j,u^d)}.
$$
In particular, recalling the local sections $\he_a, \widehat{0}$ and
$Te_a$ of $TA\to TM$ from Section \ref{subsec:corelinear}, we have
$$
\widehat{0}(\dot{x})= \dot{x}^j\frac{\partial}{\partial
x^j}\Big|_{(x^j,0)},\;\; \he_a(\dot{x})=
\dot{x}^j\frac{\partial}{\partial x^j} + \frac{\partial}{\partial
u^a}\Big|_{(x^j,0)},\;\; Te_a(\dot{x}) =
\dot{x}^j\frac{\partial}{\partial x^j}\Big|_{(x^j,\delta^d_a)},
$$
where $\dot{x}=(x^j,\dot{x}^j)\in TM$. Using \eqref{eq:secPiTA1} and
\eqref{eq:secPiTA2}, formulas \eqref{eq:lambdacore} and
\eqref{eq:lambdalinear} follow from a direct calculation.
\end{proof}

Let $(x^j,\dot{x}^j_1,\ldots,\dot{x}_k^j)$ be local coordinates on
$\oplus^kTM$, and fix $n\in \{1,\ldots,k\}$.

\begin{lemma}\label{lem:trick}
Let $\alpha \in \Omega^l(\oplus^k TM)$ be such that
$\Lie_{\frac{\partial}{\partial \dot{x}_n^j}} \alpha = 0$ $\forall
j$, and consider on $\oplus^kTM$ the local vector fields
$\dot{x}_n={\dot{x}_n^j\frac{\partial}{\partial x^j}}$,
$V^v=v^j(x)\frac{\partial}{\partial \dot{x}_n^j}$, and
$V^h=v^j(x)\frac{\partial}{\partial x^j}$ . Then
$\Lie_{V^v}i_{\dot{x}_n}\alpha = i_{V^h} \alpha$.
\end{lemma}
\begin{proof}
The proof follows from the identity $i_{[X,Y]}=\Lie_Xi_Y -
i_Y\Lie_X$ and the fact that $[v^i(x)\frac{\partial}{\partial
\dot{x}_n^i},{\dot{x}_n^j\frac{\partial}{\partial x^j}}] =
v^j\frac{\partial}{\partial x^j} - \dot{x}_n^j \frac{\partial
v^i}{\partial x^j}\frac{\partial}{\partial \dot{x}_n^i}$.
\end{proof}

We now proceed to the proof of the main result.
\begin{proof} (of Theorem~\ref{thm:mainresult})

To show that the map $\overline{\Lambda}$ in \eqref{eq:Lvbmap2} is a
Lie algebroid morphism (see e.g. \cite{Mac-book}), the only
condition to be verified is
\begin{equation}\label{eq:liealgcond}
\overline{\Lambda}([U,V]_{k})=
\Lie_{\rho_{k}(U)}\overline{\Lambda}(V)-\Lie_{\rho_{k}(V)}\overline{\Lambda}(U)
\end{equation}
for all $U,V$ sections of $\oplus_A^{k}TA\rmap \oplus^k TM$. Since
sections of type $\he_{a,n}$ (core) and $(Te_b)^k$ (linear) locally
generate the space of sections of $\oplus^k_A TA \rmap \oplus^k TM$,
it suffices to verify \eqref{eq:liealgcond} taking $U$ and $V$ to be
of these types.

\medskip
\noindent{\textit{Core-Core:}} Let us consider two core sections
$\he_{a,n}$ and $\he_{b,m}$. Since $[\he_{a,n},\he_{b,m}]_k = 0$
\eqref{eq:Tbr1}, condition \eqref{eq:liealgcond} in this case
becomes
\begin{equation}\label{eq:corecore}
\Lie_{\rho_k(\he_{a,n})}\oLambda(\he_{b,m}) -
\Lie_{\rho_k(\he_{b,m})}\oLambda(\he_{a,n}) = 0.
\end{equation}
Using \eqref{eq:rhocore} and \eqref{eq:lambdacore}, we see that
$$
\Lie_{\rho_k(\he_{a,n})}\oLambda(\he_{b,m}) =
(-1)^{n-1}\Lie_{\rho_a^i\frac{\partial}{\partial
\dot{x}_n^i}}I^{\dot{x}}_{k,m+1}I_{m-1,1}^{\dot{x}}\mu(e_b).
$$
This condition is trivially satisfied when $n=m$, so we may assume
that $n>m$ (the case $n<m$ leads to the same). Using Lemma
\ref{lem:trick}, we see that the right-hand side of the last
equation agrees with
\begin{align*}
& (-1)^{n-1}I^{\dot{x}}_{k,n+1} i_{\rho(e_a)} I^{\dot{x}}_{n-1,m+1}
I^{\dot{x}}_{m-1,1} \mu(e_b)=\\&(-1)^{n-1}(-1)^{n-2}
I^{\dot{x}}_{k,n+1} I^{\dot{x}}_{n-1,m+1} I^{\dot{x}}_{m-1,1}
i_{\rho(e_a)}\mu(e_b).
\end{align*}
Hence we obtain
$$
\Lie_{\rho_k(\he_{a,n})}\oLambda(\he_{b,m}) = - I^{\dot{x}}_{k,n+1}
I^{\dot{x}}_{n-1,m+1} I^{\dot{x}}_{m-1,1} i_{\rho(e_a)}\mu(e_b).
$$
An analogous computation  leads to
$$
\Lie_{\rho_k(\he_{b,m})}\oLambda(\he_{a,n}) = I^{\dot{x}}_{k,n+1}
I^{\dot{x}}_{n-1,m+1} I^{\dot{x}}_{m-1,1} i_{\rho(e_b)}\mu(e_a).
$$
It follows that \eqref{eq:corecore} is equivalent to
$$
i_{\rho(e_a)}\mu(e_b)=-i_{\rho(e_b)}\mu(e_a).
$$

\medskip
\noindent{\textit{Core-Linear:}} We now consider sections
$\he_{b,m}$ and $(Te_a)^k$, so that \eqref{eq:liealgcond} reads
\begin{equation}\label{eq:corelinear}
\oLambda([(Te_a)^k,\he_{b,m}]_k) =
\Lie_{\rho_k((Te_a)^k)}\oLambda(\he_{b,m}) -
\Lie_{\rho_k(\he_{b,m})}\oLambda((Te_a)^k).
\end{equation}
Using the linearity of $\Lambda$, \eqref{eq:Tbr2} and
\eqref{eq:lambdacore}, we have
\begin{align}\label{eq:cl1}
\oLambda([(Te_a)^k,\he_{b,m}]_k) &= \oLambda(C_{ab}^d\he_{d,m}) =
C_{ab}^d(-1)^{m-1}I_{k,m+1}^{\dot{x}}I_{m-1,1}^{\dot{x}}\mu(e_d)\\
&=(-1)^{m-1}I_{k,m+1}^{\dot{x}}I_{m-1,1}^{\dot{x}}\mu([e_a,e_b]).
\nonumber
\end{align}

For each fixed $n$, consider the functions $W^j_{a,n}=\frac{\partial
\rho_a^j}{\partial x^i}\dot{x}^i_n$ defined in \eqref{eq:rholinear},
noticing the following identity (of local vector fields on
$\oplus^kTM$):
\begin{equation}\label{eq:W}
W^j_{a,n}\frac{\partial}{\partial x^j}=-[\rho(e_a),\dot{x}_n],
\end{equation}
where $\dot{x}_n = \dot{x}_n^i\frac{\partial}{\partial x^i}$. Using
\eqref{eq:W} and Lemma~\ref{lem:trick}, we see that
\begin{align*}
\Lie_{\rho_k((Te_a)^k)}\oLambda(\he_{b,m}) & =\left
(\Lie_{\rho(e_a)}+
\sum_{l=1}^k\Lie_{W_{a,l}^i\frac{\partial}{\partial\dot{x}_l^i}}
\right ) (-1)^{m-1}I_{k,m+1}^{\dot{x}}I_{m-1,1}^{\dot{x}}\mu(e_b)\\
& = (-1)^{m-1}\left ( \Lie_{\rho(e_a)}I_{k,1}^{U}\mu(e_b) -
\sum_{l=1}^k I^{U}_{k,l+1}i_{[\rho(e_a),U_l]}I_{l-1,1}^{U}\mu(e_b)
\right )
\end{align*}
where $U=(U_1,\ldots,U_{k-1})=
(\dot{x}_1,\ldots,\dot{x}_{m-1},\dot{x}_{m+1},\ldots,\dot{x}_k)$. It
follows from \eqref{eq:LX} that
\begin{equation}\label{eq:cl2}
\Lie_{\rho_k((Te_a)^k)}\oLambda(\he_{b,m}) =
(-1)^{m-1}I^{\dot{x}}_{k,m+1}I^{\dot{x}}_{m-1,1}\Lie_{\rho(e_a)}\mu(e_b).
\end{equation}
Using \eqref{eq:lambdalinear} and Lemma \ref{lem:trick}, we obtain
\begin{align*}
\Lie_{\rho_k(\he_{b,m})}\oLambda((Te_a)^k) & =
\Lie_{\rho_a^i\frac{\partial}{\partial
\dot{x}^i_m}}I_{k,1}^{\dot{x}}(\dd\mu(e_a)+\nu(e_a))   =
I^{\dot{x}}_{k,m+1}i_{\rho(e_a)}I^{\dot{x}}_{m-1,1}(\dd\mu(e_a)+\nu(e_a))\\
& =
(-1)^{m-1}I_{k,m+1}^{\dot{x}}I^{\dot{x}}_{m-1,1}i_{\rho(e_a)}(\dd\mu(e_a)+\nu(e_a)).
\end{align*}
Combining this last equation with \eqref{eq:cl1} and \eqref{eq:cl2},
we see that \eqref{eq:corelinear} is equivalent to
\begin{equation}\label{eq:corelinear2}
\mu([e_a,e_b])=\Lie_{\rho(e_a)}\mu(e_b) - i_{\rho(e_b)}\dd\mu(e_a) -
i_{\rho(e_b)}\nu(e_a).
\end{equation}

\medskip
\noindent{\textit{Linear-Linear:}} We finally consider condition
\eqref{eq:liealgcond} for two linear sections:
\begin{equation}\label{eq:linearlinear}
\oLambda([(Te_a)^k,(Te_b)^k]_k)=\Lie_{\rho_k((Te_a)^k)}\oLambda((Te_b)^k)
-\Lie_{\rho_k((Te_b)^k)}\oLambda((Te_a)^k).
\end{equation}
Using \eqref{eq:Tbr3} and the linearity of $\Lambda$, we have
\begin{align}
\oLambda([(Te_a)^k,(Te_b)^k]_k) &=
C_{ab}^d\oLambda((Te_d)^k)+\sum_{n=1}^k \dd C_{ab}^d(\dot{x}_n) \oLambda(\he_{d,n})\nonumber \\
& = C_{ab}^d I_{k,1}^{\dot{x}}(\dd\mu(e_d)+\nu(e_d)) + \sum_{n=1}^k
(-1)^{n-1} \dd C_{ab}^d(\dot{x}_n)
I^{\dot{x}}_{k,n+1}I^{\dot{x}}_{n-1,1}\mu(e_d).\label{eq:ll1}
\end{align}
It follows from \eqref{eq:alpha} (also using that
$I^{\dot{x}}_{k,1}\mu(e_d)=0$, since $\mu(e_d)$ is a $(k-1)$-form)
that
\begin{align*}
I^{\dot{x}}_{k,1} C_{ab}^d \dd\mu(e_d) &=  I^{\dot{x}}_{k,1}
\dd(C_{ab}^d \mu(e_d)) - I^{\dot{x}}_{k,1} (\dd C_{ab}^d \wedge \mu(e_d))\\
& = I^{\dot{x}}_{k,1} \dd(C_{ab}^d \mu(e_d)) -
\sum_{n=1}^k(-1)^{n+1}\dd
C_{ab}^d(\dot{x}_n)I^{\dot{x}}_{k,n+1}I^{\dot{x}}_{n-1,1}\mu(e_d).
\end{align*}
Comparing with \eqref{eq:ll1}, we conclude that
\begin{align}\label{eq:ll2}
\oLambda([(Te_a)^k,(Te_b)^k]_k) & =
I^{\dot{x}}_{k,1}(\dd\mu(C_{ab}^d
e_d) + \nu(C_{ab}^d e_d))\\
& = I^{\dot{x}}_{k,1}(\dd\mu([e_a,e_b]) + \nu([e_a,e_b])). \nonumber
\end{align}
Using Lemma \ref{lem:trick}, \eqref{eq:W} and \eqref{eq:LX}, we
directly obtain
\begin{align}\label{eq:ll3}
\Lie_{\rho((Te_a)^k)}\oLambda((Te_b)^k) & = \left ( \Lie_{\rho(e_a)}
+ \sum_{n=1}^k \Lie_{W_{a,n}^i\frac{\partial}{\partial
\dot{x}_n^i}}\right )I^{\dot{x}}_{k,1}(\dd\mu(e_b)+\nu(e_b))\\
& =
I_{k,1}^{\dot{x}}\Lie_{\rho(e_a)}(\dd\mu(e_b)+\nu(e_b)).\nonumber
\end{align}
Similarly
\begin{equation}\label{eq:ll4}
\Lie_{\rho((Te_b)^k)}\oLambda((Te_a)^k) =
I_{k,1}^{\dot{x}}\Lie_{\rho(e_b)}(\dd\mu(e_a)+\nu(e_a)).
\end{equation}
Combining \eqref{eq:ll2}, \eqref{eq:ll3} and \eqref{eq:ll4}, we see
that \eqref{eq:linearlinear} is equivalent to
$$
\dd\mu([e_a,e_b])+\nu([e_a,e_b])=\Lie_{\rho(e_a)}(\dd\mu(e_b)+\nu(e_b))
- \Lie_{\rho(e_b)}(\dd\mu(e_a)+\nu(e_a)).
$$
We may assume that \eqref{eq:corelinear2} holds, in which case one
can directly check that the last equation is equivalent to
$$
\nu([e_a,e_b])=\Lie_{\rho(e_a)}\nu(e_b) - i_{\rho(e_b)}\dd\nu(e_a).
$$
\end{proof}

%%%%%%%%%%%%%%%%%%%%%%%%%%%%%%%%%%%%%%%%%%%%%%%%%%%%%%%%%%%%%%%%%%%%%%%%5
\section{Infinitesimal description of multiplicative
forms}\label{sec:mult}

In this section, we relate IM-forms on Lie algebroids with
multiplicative forms on Lie groupoids. Let $\Grd$ be a Lie groupoid
over $M$, with source and target maps denoted by $\sour$, $\tar :
\Grd \rmap M$, respectively, multiplication $m:\Grd^{(2)}\rmap
\Grd$, and unit map $\epsilon: M \rmap \Grd$ (that we often use to
view $M$ as a submanifold of $\Grd$). The Lie algebroid of $\Grd$ is
denoted by $A(\Grd)$, or simply $A$ if there is no risk of
confusion; see Section \ref{subsec:ident}.

A $k$-form $\alpha \in \Omega^k(\Grd)$ is called
\textbf{multiplicative} if
\begin{equation}\label{eq:mult}
m^*\alpha = pr_1^*\alpha + pr_2^*\alpha,
\end{equation}
where $pr_1,pr_2: \Grd^{(2)}\rmap \Grd$ are the natural projections.
Alternatively, one may define multiplicative forms in terms of a
natural groupoid structure on $T\Grd$ over $TM$, known as the
\textit{tangent groupoid}, see e.g. \cite{Mac-book}; it has source
(resp. target) map $T\sour: T\Grd\rmap TM$ (resp. $T\tar: T\Grd\rmap
TM$), multiplication $Tm:(T\Grd)^{(2)}=T\Grd^{(2)}\rmap T\Grd$, and
unit map $T\epsilon :TM \rmap T\Grd$. This groupoid structure can be
naturally extended to the direct sum $\oplus^k_\Grd T\Grd$, $k\geq
1$, making it a Lie groupoid over $\oplus^kTM$, with source (resp.
target) map $\oplus^kT\sour$ (resp. $\oplus^kT\tar$), multiplication
map $\oplus^kTm$, etc.

Let $\alpha\in\Omega^k(\Grd)$, and let us consider the associated
map
\begin{equation}\label{eq:alphabar}
\overline{\alpha}: \oplus^k_\Grd T\Grd \rmap \mathbb{R},\;\;\;
\overline{\alpha}(U_1,\ldots,U_k)= i_{U_k}\ldots i_{U_1}\alpha.
\end{equation}
The following observation is immediate from \eqref{eq:mult}.

\begin{lemma}\label{lem:multip}
$\alpha$ is multiplicative if and only if $\overline{\alpha}$ is a
groupoid morphism. (Here $\mathbb{R}$ is viewed as an additive
group.)
\end{lemma}

We denote the space of multiplicative $k$-forms on $\Grd$ by
$\Omega^k_{\mathrm{mult}}(\Grd)$ .

%%%%%%%%%%%%%%%%%%%%%%%%%%%%%%%%%%%%%%%%%%%%%%%%%%%%%%%%%%%%%%%%%%%%%
\subsection{From multiplicative to IM forms}

Let $\Grd$ be a Lie groupoid over $M$, and consider the tangent lift
operation $\Omega^k(\Grd)\to \Omega^{k}(T\Grd)$, $\alpha\mapsto
\alpha_T$, recalled in Section \ref{subsec:tlift}. Using the natural
inclusion $\iota_A: A=A\Grd \hookrightarrow T\Grd$, we define a map
\begin{equation}\label{eq:liemap}
\LF: \Omega^k(\Grd)\rmap \Omega^k(A), \;\; \alpha \mapsto
\LF(\alpha) = \iota_A^* \alpha_T.
\end{equation}

Given $\alpha\in \Omega^k(\Grd)$,  let us consider the associated
bundle maps $\mu: A\rmap \wedge^{k-1} T^*M$ and $\nu: A \rmap
\wedge^k T^*M$,
\begin{align}
\SP{\mu(u), X_1\wedge\ldots\wedge
X_{k-1}} & = \alpha(u,X_1,\ldots,X_{k-1}), \label{eq:mu}\\
\SP{\nu(u), X_1\wedge\ldots\wedge X_{k}} & =
\dd\alpha(u,X_1,\ldots,X_{k}) \label{eq:nu},
\end{align}
for $X_1,\ldots,X_k \in TM$ and $u\in A$ (here we use the natural
inclusions $TM\hookrightarrow T\Grd|_M$ and $A\hookrightarrow
T\Grd|_M$).

\begin{lemma}\label{lem:Lie1}
The $k$-form $\LF(\alpha) \in \Omega^k(A)$ is linear and satisfies
$$
\LF(\alpha) = \dd\Lambda_\mu + \Lambda_\nu.
$$
\end{lemma}
\begin{proof}
Let $\beta\in \Omega^l(\Grd)$ be any $l$-form on $\Grd$, and let us
consider the $l-1$-form on $A$ given by $\iota_A^*\tau(\beta)$ (see
\eqref{eq:tau}), i.e.,
\begin{align*}
\iota_A^*\tau(\beta)|_u = (T\iota_A|_u)^t
\tau(\beta)|_{\iota_A(u)}&=
(T\iota_A|_u)^t(Tp_\Grd|_{\iota_A(u)})^ti_{\iota_A(u)}\beta\\& =
(T(p_\Grd\circ \iota_A)|_u)^t i_{\iota_A(u)}\beta.
\end{align*}
From the commutative diagram
$$
\xymatrix{ A  \ar[r]^{\iota_A} \ar[d]_{q_A} &
T\Grd \ar[d]^{p_\Grd}  \\
M \ar[r]_{\epsilon} & \Grd, }
$$
we see that
$$
\iota_A^*\tau(\beta)|_u =
(Tq_A|_u)^t(T\epsilon|_{q_A(u)})^ti_{\iota_A(u)}\beta.
$$
It immediately follows (see \eqref{eq:lambamu}) that
$$
\iota_A^*\tau(\alpha) = \Lambda_\mu,\;\;\;
\iota_A^*\tau(\dd\alpha)=\Lambda_\nu.
$$
Using \eqref{eq:cartanlike}, we see that
$$
\Lambda = \iota_A^*(\dd\tau(\alpha) +
\tau(\dd\alpha))=\dd\iota_A^*\tau(\alpha) + \iota_A^*\tau(\dd\alpha)
= \dd\Lambda_\mu + \Lambda_\nu.
$$
\end{proof}

Recall that any groupoid morphism $\psi: \Grd_1 \rmap \Grd_2$
defines a Lie algebroid morphism $\LF(\psi): A\Grd_1 \rmap A\Grd_2$
that fits into the diagram
\begin{equation}\label{eq:liemorphism}
\xymatrix{ T\Grd_1  \ar[r]^{T\psi}  & T\Grd_2  \\
A\Grd_1 \ar[r]_{\LF(\psi)} \ar[u]^{\iota_{A_1}}&  A\Grd_2
\ar[u]_{\iota_{A_2}} }
\end{equation}
When $\alpha\in \Omega^k(\Grd)$ is multiplicative, we saw in Lemma
\ref{lem:multip} that $\overline{\alpha}: \oplus^k_\Grd T\Grd \rmap
\mathbb{R}$ is a groupoid morphism; we consider its infinitesimal
counterpart,
$$
\LF(\overline{\alpha}):A(\oplus^k_\Grd T\Grd)\rmap \mathbb{R},
$$
where now $\mathbb{R}$ is viewed as the trivial Lie algebroid over a
point. The natural projection $p_\Grd:T\Grd\rmap \Grd$ is a groupoid
morphism, and there is a canonical identification of Lie
algebroids\footnote{The Lie algebroid $A(T\Grd)\to TM$ is a
$\mathcal{VB}$-algebroid \cite[Sec. 2.1]{GM} with respect to the
vector bundle structure $\LF(p_\Grd):A(T\Grd)\to A$; the algebroid
structure on $A(T\Grd)$ can be extended to $\prod_{\LF(p_\Grd)}^k
A(T\Grd)$ in terms of core and linear sections, just as described in
Section~\ref{subsec:Liealg}.}
$$
A(\oplus^k_\Grd T\Grd) = \prod_{\LF(p_\Grd)}^k A(T\Grd).
$$
Our next goal is to compare the following two maps:
$$
\overline{\LF(\alpha)}: \oplus^k_A T(A\Grd)\rmap \mathbb{R}\;\;\;
\mbox{ and }\;\;\; \LF(\overline{\alpha}): \prod_{\LF(p_\Grd)}^k
A(T\Grd)\rmap \mathbb{R}.
$$

The involution $J_\Grd:T(T\Grd)\rmap T(T\Grd)$ (see \eqref{eq:J})
defines an identification of Lie algebroids $j_\Grd: T(A\Grd) \rmap
A(T\Grd)$ via the diagram
\begin{equation}\label{eq:diagJ}
\xymatrix{ T(A\Grd)  \ar[r]^{j_\Grd} \ar[d]_{T\iota_{A\Grd}} & A(T\Grd) \ar[d]^{\iota_{A(T\Grd)}}  \\
T(T\Grd) \ar[r]_{J_\Grd} &  T(T\Grd) }
\end{equation}
Note that the property $Tp_\Grd\circ J_\Grd = p_{T\Grd}$ implies
that
$$
\LF(p_\Grd)\circ j_\Grd = p_A.
$$
As a result, we have a natural identification of Lie algebroids,
\begin{equation}\label{eq:jk}
j_\Grd^{(k)} : \oplus^k_A TA = \prod^k_{p_A} T(A\Grd)
\stackrel{\sim}{\rmap} \prod^k_{\LF(p_\Grd)}A(T\Grd),
\end{equation}
fitting into the diagram
\begin{equation}\label{eq:jkdiag}
\xymatrix{ \prod^k_{p_A}T(A\Grd)  \ar[r]^{j_\Grd^{(k)}} \ar[d]_{
(T\iota_{A\Grd})^k} &
\prod_{\LF(p_\Grd)}^k A(T\Grd) \ar[d]^{ (\iota_{A(T\Grd)})^k}  \\
\prod_{p_{T\Grd}}^k T(T\Grd) \ar[r]_{J_\Grd^{(k)}} &
\prod_{Tp_\Grd}^k T(T\Grd). }
\end{equation}

\begin{lemma}\label{lem:lie} Let $\alpha \in \Omega^k(\Grd)$ be
multiplicative. Then
\begin{equation}
\LF(\overline{\alpha})\circ j_\Grd^{(k)} = \overline{\LF(\alpha)}.
\end{equation}
In particular, $\overline{\LF(\alpha)}:\oplus^k_A T(A\Grd)\rmap
\mathbb{R}$ is a Lie algebroid morphism.
\end{lemma}

\begin{proof}
By definition, $\LF(\overline{\alpha})=\dd \overline{\alpha}\circ
(\iota_{A(T\Grd)})^k$, and using \eqref{eq:jkdiag}
 and \eqref{eq:alphaT} we obtain
\begin{align*}
\LF(\overline{\alpha})\circ j_\Grd^{(k)} = \dd \overline{\alpha}\circ
(\iota_{A(T\Grd)})^k \circ j_\Grd^{(k)} &= \dd \overline{\alpha}\circ
J_\Grd^{(k)} \circ (T\iota_{A\Grd})^k \\ & =
\overline{\alpha_T}\circ (T\iota_{A\Grd})^k =
\overline{\iota_A^*\alpha_T}.
\end{align*}
\end{proof}

\begin{proposition}\label{prop:lie}
Let $\alpha \in \Omega^k(\Grd)$ be multiplicative, and let $\mu$ and
$\nu$ be defined as in \eqref{eq:mu} and \eqref{eq:nu}. Then
$(\mu,\nu)$ is an IM $k$-form on $A\Grd$.
\end{proposition}

\begin{proof}
The result is a direct consequence of Lemmas~\ref{lem:Lie1},
\ref{lem:lie}, and Theorem~\ref{thm:mainresult}.
\end{proof}

%%%%%%%%%%%%%%%%%%%%%%%%%%%%%%%%%%%%%%%%%%%%%%%%%%%%%%%%%%%%%%%%%%%%%
\subsection{Integration of IM forms}

Let $\Grd$ be a Lie groupoid over $M$, with Lie algebroid $A=A\Grd$.
Assume that $\Grd$ is \textit{source-simply-connected} (i.e., the
$\sour$-fibres are connected with trivial fundamental group), so
that $\oplus^k_\Grd T\Grd$ is also a source-simply-connected
groupoid\footnote{Given any $X=(X_1,\ldots,X_k)\in \oplus^kT_xM$,
the projection $(p_\Grd)^k:\oplus^k_\Grd T\Grd \rmap \Grd$ makes the
source fibre $((T\sour)^k)^{-1}(X) \subseteq T\Grd$ into an affine
bundle over the source fibre $\sour^{-1}(x)\subseteq \Grd$}. Let
$\Lambda \in \Omega^k(A)$ be a $k$-form on $A$ for which
$\overline{\Lambda}: \oplus^k_A TA \rmap \mathbb{R}$ is a Lie
algebroid morphism.

\begin{lemma}\label{lem:int}
There is a unique multiplicative $k$-form $\alpha \in
\Omega^k(\Grd)$ such that $\LF(\alpha)=\Lambda$ (see
\eqref{eq:liemap}).
\end{lemma}

\begin{proof}
Since $\oLambda$ is a morphism of Lie algebroids, the identification
\eqref{eq:jk} also leads to a Lie algebroid morphism
\begin{equation}\label{eq:morp}
\overline{\Lambda}\circ (j_\Grd^{(k)})^{-1} =
\prod^k_{\LF(p_\Grd)}A(T\Grd)\cong A(\oplus^k T\Grd) \rmap
\mathbb{R}.
\end{equation}

As $\oplus^k_\Grd T\Grd$ is a source-simply-connected groupoid, we
can use Lie's second theorem (see e.g. \cite{Mac-book}) to obtain a
unique groupoid morphism
\begin{equation}\label{eq:int}
I_\Lambda : \oplus^k_\Grd T\Grd \rmap \mathbb{R}
\end{equation}
integrating the morphism \eqref{eq:morp}, i.e., such that
$\LF(I_\Lambda) = \overline{\Lambda}\circ (j_\Grd^{(k)})^{-1}$. To
check that $I_\Lambda = \overline{\alpha}$, for $\alpha \in
\Omega^k(\Grd)$, it suffices to verify that the following conditions
hold:
 \begin{align}
I_\Lambda(U_1,\ldots,U_i,\ldots,U_j,\ldots,U_k) = & -
I_{\Lambda}(U_1,\ldots,U_j,\ldots,U_i,\ldots,U_k),\label{eq:I1}\\
I_\Lambda(U_1,\ldots,U_{i-1},cU_i,U_{i+1},\ldots,U_k)=& c
I_\Lambda(U_1,\dots,U_k),\label{eq:I2}\\
I_{\Lambda}(U_1,\ldots,U_{i-1},U_i+U'_i,U_{i+1},\ldots,U_k) =&
I_\Lambda(U_1,\dots,U_i,\dots,U_k) + \label{eq:I3}\\ &
I_\Lambda(U_1,\dots,U'_i,\ldots,U_k),\nonumber
\end{align}
for all $U_i,U'_i\in T_g\Grd$, $g\in \Grd$, $c\in \mathbb{R}$, where
$1\leq i < j \leq k$.  As we now show, all conditions can be
verified with the same type arguments (cf. \cite{Mac-Xu}).

To prove that \eqref{eq:I1} holds, one directly checks that the map
$ I_\Lambda^{(ij)}: \oplus^k_\Grd T\Grd\rmap \mathbb{R}$,
$$
I_\Lambda^{(ij)}(U_1,\ldots,U_i,\ldots,U_j,\ldots,U_k) : =  -
I_{\Lambda}(U_1,\ldots,U_j,\ldots,U_i,\ldots,U_k),
$$
is a groupoid morphism, and $\LF(I_\Lambda^{(ij)}):
\prod_{\LF(p_\Grd)}^kA(T\Grd) \rmap \mathbb{R}$ satisfies
\begin{align*}
\LF(I_\Lambda^{(ij)})(V_1,\ldots,V_i,\ldots,V_j,\ldots,V_k)&=-
\LF(I_\Lambda)(V_1,\ldots,V_j,\ldots,V_i,\ldots,V_k)\\
&=-{\Lambda}\circ
(j_\Grd^{(k)})^{-1}(V_1,\ldots,V_j,\ldots,V_i,\ldots,V_k)\\
%(j_\Grd^{-1}(V_1),\ldots,j_\Grd^{-1}(V_j),\ldots,j_\Grd^{-1}(V_i),\ldots,j_\Grd^{-1}(V_k))\\
&=\LF(I_\Lambda)(V_1,\ldots,V_i,\ldots,V_j,\ldots,V_k),
\end{align*}
since $\Lambda$ is skew-symmetric. So
$\LF(I_\Lambda^{(ij)})=\LF(I_\Lambda)$, and the uniqueness of
integration in Lie's second theorem implies that
$I_\Lambda^{(ij)}=I_\Lambda$, which is \eqref{eq:I1}.

Similarly, for a fixed $c\in \mathbb{R}$, one can directly show that
both the left and right-hand sides of \eqref{eq:I2} define groupoid
morphisms $\oplus^k_\Grd T\Grd \rmap \mathbb{R}$, whose
infinitesimal counterparts agree at the level of Lie algebroids due
to the multilinearity of $\Lambda$. Then \eqref{eq:I2} follows again
by the uniqueness part of Lie's second theorem.

The last condition \eqref{eq:I3} can be treated in a completely
analogous way, by first noticing that both sides of \eqref{eq:I3}
define groupoid morphisms $\oplus^{k+1}_\Grd T\Grd \rmap
\mathbb{R}$, where now we need an extra copy of $T\Grd$ for $U_i'$.
Again, these morphisms agree at the infinitesimal level due to the
multilinearity of $\Lambda$, and hence agree globally.

The fact that $\alpha$ is multiplicative follows from
Lemma~\ref{lem:multip}, and the equality $\Lambda = \LF(\alpha)$ is
a consequence of Lemma~\ref{lem:lie}.
\end{proof}

A direct consequence of Lemmas~\ref{lem:lie} and \ref{lem:int} is
that the map
$$
\Omega^k_{\mathrm{mult}}(\Grd) \to \Omega^k(A),\;\;\; \alpha\mapsto
\LF(\alpha),
$$
is a bijection onto the subspace of $k$-forms $\Lambda \in
\Omega^k(A)$ such that $\oLambda : \oplus^k_A TA \to \mathbb{R}$ is
a morphism of Lie algebroids. By the correspondence in
Theorem~\ref{thm:mainresult}, this bijection can be alternatively
phrased in terms of IM-forms on $A$:

\begin{theorem}\label{thm:1-1}
Let $\Grd$ be a source-simply-connected Lie groupoid over $M$ with
Lie algebroid $A\rmap M$. For each positive integer $k$, there is a
1-1 correspondence
\begin{equation}\label{eq:1-1}
\Omega_{\mathrm{mult}}^k(\Grd)\rmap \Omega^k_{\mathrm{IM}}(A), \;\;
\alpha \mapsto (\mu,\nu),
\end{equation}
where $\mu$, $\nu$ are given by
\begin{align}
\SP{\mu(u), X_1\wedge\ldots\wedge
X_{k-1}} & = \alpha(u,X_1,\ldots,X_{k-1}), \label{eq:mu1}\\
\SP{\nu(u), X_1\wedge\ldots\wedge X_{k}} & =
\dd\alpha(u,X_1,\ldots,X_{k}) \label{eq:nu1}.
\end{align}
\end{theorem}

\begin{proof}
The result follows from Lemma~\ref{lem:Lie1} and
Theorem~\ref{thm:mainresult}.
\end{proof}

%%%%%%%%%%%%%%%%%%%%%%%%%%%%%%%%%%%%%%%%%%%%%%%%%%%%%%%%%%%%%%%%%%%

The following is a simple example of correspondence in
Theorem~\ref{thm:1-1}.

\begin{example}
Let us equip $A=T^*M\to M$ with the trivial Lie algebroid structure
(both anchor and bracket are identically zero), so we may identify
$\Grd=T^*M$ (with groupoid multiplication given by fibrewise
addition). Fixing $\mu=\mathrm{Id}:T^*M\to T^*M$, then any
vector-bundle map $\nu:T^*M\to \wedge^2 T^*M$ defines an IM 2-form
$(\mu,\nu)$. When $\nu=0$, then $(\mu,\nu)$ corresponds under
\eqref{eq:1-1} to the canonical symplectic form $\omega_{can}$ on
$\Grd=T^*M$; for an arbitrary $\nu$, the corresponding
multiplicative 2-form is given, at each $g=(q^j,p_j)\in T^*M$, by
$$
\omega|_{g}= \omega_{can}|_g + c_M^*\nu(g)
$$
where $c_M:T^*M\to M$ is the natural projection.
\end{example}

Let us list some immediate consequences of Theorem~\ref{thm:1-1},
illustrating how the correspondence \eqref{eq:1-1} restricts to
subclasses of multiplicative and IM forms:

\begin{itemize}
\item[(a)] Let $\eta \in \Omega^{k}(M)$. Following Example \ref{ex:IM1}, we know that $(\mu,\nu)$, where
$\mu(u)=-i_{\rho(u)}\eta$ and $\nu(u)=-i_{\rho(u)}\dd\eta$, defines
an IM $k$-form. One directly verifies that the corresponding
multiplicative $k$-form is $\alpha = \sour^*\eta - \tar^*\eta$.

\item[(b)] Let $\phi \in \Omega^{k+1}(M)$ be a closed
$k+1$-form. Then Theorem~\ref{thm:1-1} gives a bijective
correspondence between IM $k$-forms on $A$ relative to $\phi$ (i.e.,
$\nu(u)=-i_{\rho(u)}\phi$, see Example~\ref{ex:IM2}) and
multiplicative $k$-forms $\alpha$ satisfying $\dd\alpha =
\sour^*\phi - \tar^*\phi$. To verify this fact, just notice that
$\dd\alpha$ is a multiplicative $(k+1)$-form corresponding to an IM
$(k+1)$-form of the type discussed in item (a). This recovers
\cite[Thm.~2.5]{bcwz} when $k=2$ (cf. \cite{BCO}), as well as
\cite[Thm.~2]{AC} when $k$ is an arbitrary positive integer.

\item[(c)] Let $\alpha\in \Omega_{\mathrm{mult}}^{k}(\Grd)$
be a given closed multiplicative $k$-form, with associated IM
$k$-form $(\mu_\alpha,\nu_\alpha)$ (note that $\nu_\alpha=0$,
necessarily). It follows from Theorem~\ref{thm:1-1} that there is a
1-1 correspondence between multiplicative $(k-1)$-forms $\theta$
with $\dd\theta=\alpha$ and vector-bundle maps $\mu: A\to
\wedge^{k-2}T^*M$ satisfying, for all $u,v \in \Gamma(A)$,
$$
\qquad i_{\rho(u)}\mu(v) = -i_{\rho(v)}\mu(u),\;\;\;\;
\mu([u,v])=\Lie_{\rho(u)}\mu(v)-i_{\rho(v)}d\mu(u)-i_{\rho(v)}\mu_\alpha(u).
$$
The reason is that $(\mu,\mu_\alpha)$ is the IM $(k-1)$-form
associated with $\theta$ (note that $\mu_\alpha$ satisfies
\eqref{eq:IM3} as a result of $(\mu_\alpha,\nu_\alpha)$ being an IM
$k$-form for $\alpha$ and $\nu_\alpha$ being zero). This
correspondence is the content of \cite[Thm.~3]{AC}.

\end{itemize}

For $k=2$, one has further refinements of Theorem~\ref{thm:1-1}
based on the general study of multiplicative 2-forms carried out in
\cite[Section~4]{bcwz}, leading to natural generalizations of
twisted Poisson and Dirac structures (in the sense of \cite{SW}). On
the vector bundle $TM\oplus T^*M$, consider the pairing
$\SP{(X,\alpha),(Y,\beta)}=\beta(X)+\alpha(Y)$, and the natural
projections $pr_T:TM\oplus T^*M\to TM$ and $pr_{T^*}:TM\oplus
T^*M\to T^*M$. As in Theorem~\ref{thm:1-1}, we denote by $\Grd$ the
source-simply-connected Lie groupoid with Lie algebroid  $A\to M$.

\begin{itemize}

\item[(d)] Given an IM 2-form $(\mu,\nu)$ on $A$, we
consider the vector-bundle map $ (\rho,\mu): A \to TM\oplus T^*M$
and its image
\begin{equation}\label{eq:L}
L=\{(\rho(u),\mu(u))\;|\; u\in A\}\subset TM\oplus T^*M,
\end{equation}
which is a subbundle whenever it has constant rank. By
\eqref{eq:IM1}, over each point of $M$, $L$ is isotropic with
respect to the pairing $\SP{\cdot,\cdot}$. It follows from
\cite[Cor.~4.8]{bcwz} that, under the assumption that
$\dim(\Grd)=2\dim(M)$, the correspondence \eqref{eq:1-1} restricts
to a bijection be\-tween multiplicative 2-forms $\omega$ such that
\begin{equation}\label{eq:ker}
\ker(T\sour)_x\cap \ker(\omega)_x\cap \ker(T\tar)_x = \{0\},\;\;\;\;
\forall x\in M,
\end{equation}
and IM 2-forms $(\mu,\nu)$ for which $L=L^\perp$ (i.e., $L$ is
\textit{lagrangian} with respect to $\SP{\cdot,\cdot}$; in
particular, it is a subbundle with $\mathrm{rank}(L)=\dim(M)$) and
\begin{equation}\label{eq:AL}
(\rho,\mu):A \rmap L\subset TM\oplus T^*M
\end{equation}
is an isomorphism of vector bundles. Moreover, $\tar:\Grd\to M$
relates $\omega$ and $L$ as a \textit{forward Dirac map} (see, e.g.,
\cite[Sec.~2.1]{bcwz}). If we define $\nu_L:L\to \wedge^2T^*M$ by
$\nu_L((\rho(u),\mu(u)))=\nu(u)$, then the identification
\eqref{eq:AL} induces a Lie algebroid structure on $L$ with anchor
$pr_T|_L$ and bracket on $\Gamma(L)$ given by
\begin{equation}\label{eq:nubrk}
[(X,\alpha),(Y,\beta)]_L :=([X,Y],\Lie_X\beta - i_Yd\alpha -
i_Y\nu_L(X,\alpha)).
\end{equation}

Conversely, if a lagrangian subbundle $L\subset TM\oplus T^*M$ is
equipped with $\nu_L:L\to \wedge^2T^*M$ for which \eqref{eq:nubrk}
is a Lie bracket on $\Gamma(L)$, then $A=L$ is a Lie algebroid with
anchor $pr_T|_L$; if $\nu_L$
 also satisfies \eqref{eq:IM3}, then $(pr_{T^*}|_L,\nu_L)$ is an IM
2-form. Then, under the bijection \eqref{eq:1-1},
$(pr_{T^*}|_L,\nu_L)$ corresponds  to a multiplicative 2-form
$\omega$ on $\Grd$ satisfying \eqref{eq:ker}. Taking $\nu_L$ to be
of type $\nu_L(X,\alpha)=-i_X\phi$ for a closed 3-form $\phi\in
\Omega^3(M)$, we recover the integration of twisted Dirac structures
by twisted presymplectic groupoids of \cite[Sec.~2]{bcwz}.

\item[(e)] When a multiplicative 2-form is nondegenerate, then
$\dim(\Grd)=2\dim(M)$ automatically (see, e.g.,
\cite[Lem.~3.3]{bcwz}), and \eqref{eq:ker} trivially holds. Under
\eqref{eq:1-1}, this situation corresponds to the case where the IM
2-form $(\mu,\nu)$ is such that $\mu: A\to T^*M$ is an isomorphism;
in other words, $L$ is the graph of a bivector field $\pi$ on $M$:
$L=\{(i_\alpha\pi,\alpha)\;|\; \alpha\in T^*M\}$. Following (d)
above, the fact that $\Gamma(L)$ is closed under the bracket
\eqref{eq:nubrk} is expressed by the compatibility condition
$$
\frac{1}{2}[\pi,\pi](\alpha,\beta,\gamma)= \nu_L((i_\alpha\pi,
\alpha))(i_\beta\pi,i_\gamma\pi)
$$
for all $\alpha,\beta,\gamma \in \Omega^1(M)$. When $\nu_L$ is of
type $\nu_L(X,\alpha)=-i_X\phi$ for a closed 3-form $\phi\in
\Omega^3(M)$, one recovers twisted Poisson structures, and
\eqref{eq:1-1} gives their integration to twisted symplectic
groupoids (cf. \cite{CX}).

\end{itemize}

\begin{remark}[Higher Dirac structures]
Just as Dirac structures are special cases of IM 2-forms, the {\em
higher} Dirac structures of \cite{Zam} are particular cases of
higher-degree IM $k$-forms; so Theorem~\ref{thm:1-1} can be used to
integrate higher Dirac structures as well (c.f.
\cite[Prop.~3.7]{Zam}).
\end{remark}

\begin{remark}[Path-space construction]
We now relate Theorem~\ref{thm:1-1} to the {\em path-space} approach
to integration found in \cite{bcwz,catfel,CX}. For an integrable Lie
algebroid $A$, there is an explicit model for its integrating
source-simply-connected Lie groupoid $\Grd(A)$ \cite{CF,severa};
namely, $\Grd(A)$ is the quotient $PA/\sim$, where $PA$ is the
subspace of {\em $A$-paths} in the Banach manifold of all $C^1$
paths from the interval $I=[0,1]$ into $A$ with $C^2$ projection to
$M$, and $\sim$ is the equivalence relation defined by {\em
$A$-homotopies}, see \cite{CF}. If $\Lambda \in \Omega^k(A)$ is a
$k$-form on $A$, we use the evaluation map $ev: PA \times I \to A,
(a(\cdot),t) \mapsto a(t)$ to define a $k$-form $\widetilde{\Lambda}
\in \Omega^k(PA)$ by the formula
\begin{equation} \label{eq: tildeLambda}
 \widetilde{\Lambda}= \int_I dt \ ev^*\Lambda.
\end{equation}
One verifies that, if $\Lambda \in \Omega^k(A)$ is a linear
$k$-form, then the $k$-form $\widetilde{\Lambda}$ \eqref{eq:
tildeLambda} is \emph{basic} with respect to the quotient projection
$q:PA \to \Grd(A)$, i.e., $\widetilde{\Lambda} = q^*\alpha$ for
 $\alpha \in \Omega^k(\Grd(A))$, if and only if the induced map
$\bar{\Lambda}$ \eqref{eq:lbar}  is a Lie algebroid morphism (i.e.,
if $\Lambda$ is an IM $k$-form). In this case, the multiplicative
$k$-form obtained by integration of the morphism $\bar{\Lambda}$ to
$\Grd(A)$ (see Lemma~\ref{lem:int}) agrees with $\alpha$, showing
how the correspondence in Thm.~\ref{thm:1-1} is viewed in the light
of the path-space method (generalizing the approach in
\cite[Sec.~5]{bcwz}).
\end{remark}

%%%%%%%%%%%%%%%%%%%%%%%%%%%%%%%%%%%%%%%%%%%%%%%%%%%%%%%%%%%%%%%%%%%%%
\section{Relation with the Weil algebra and Van Est
isomorphism}\label{sec:VE}

This section clarifies how linear and IM-forms on Lie algebroids fit
into the Weil algebra of \cite[Sec.~3]{AC}, and how the
infinitesimal description of multiplicative forms relates to the
general Van Est isomorphism of \cite[Sec.~4]{AC}.

Let $A$ be a Lie algebroid over $M$. We consider the associated
\textit{Weil algebra} $W(A)$ as in \cite[Sec.~3]{AC}, which is a
bi-graded differential algebra. The space of elements of degree
$(p,k)$ is denoted by $W^{p,k}(A)$, and the differential on $W(A)$
is a sum of differentials $d^h + d^v$, where
$$
d^h: W^{p,k}(A)\rmap W^{p+1,k}(A),\;\;\; d^v:W^{p,k}(A)\rmap
W^{p,k+1}(A).
$$
We will be mostly concerned with $W^{p,k}(A)$ for $p=0,1,2$.

For $p=0$, we have $W^{0,k}(A)=\Omega^k(M)$. An element $\Lambda_W
\in W^{1,k}(A)$ is given by a pair $((\Lambda_W)_0,(\Lambda_W)_1)$,
where
\begin{equation}\label{eq:LWd1}
(\Lambda_W)_0 : \Gamma(A)\rmap \Omega^k(M),\;\;\; (\Lambda_W)_1 \in
\Omega^{k-1}(M,A^*),
\end{equation}
subject to the compatibility condition
\begin{equation}\label{eq:compat}
(\Lambda_W)_0(f u) = f(\Lambda_W)_0(u) - df\wedge (\Lambda_W)_1(u),
\end{equation}
for $f\in C^\infty(M)$, $u\in \Gamma(A)$, and $(\Lambda_W)_1$ viewed
as a $C^\infty(M)$-linear map
\begin{equation}
(\Lambda_W)_1: \Gamma(A)\rmap \Omega^{k-1}(M).
\end{equation}
An element $c_W \in W^{2,k}(A)$ is a triple $((c_W)_0,
(c_W)_1,(c_W)_2)$, where
\begin{align}\label{eq:L2}
& (c_W)_0 : \Gamma(A)\times \Gamma(A) \rmap \Omega^k(M),\;\; (c_W)_1
: \Gamma(A)\rmap \Omega^{k-1}(M,A^*),
\\\nonumber &
(c_W)_2 \in \Omega^{k-1}(M,S^2(A^*)),
\end{align}
and such that $(c_W)_0$ is skewsymmetric and $\mathbb{R}$-bilinear,
subject to suitable compatibility conditions (extending
\eqref{eq:compat}) that we will not need explicitly.

We need to recall the expression for $d^h$ restricted to
$W^{1,k}(A)$. By definition (see \cite[Sec.~3.1]{AC}), for
$\Lambda_W = ((\Lambda_W)_0,(\Lambda_W)_1)\in W^{1,k}(A)$,
$d^h\Lambda_W \in W^{2,k}(A)$ is given by (cf. \eqref{eq:L2})
\begin{align}
&(d^h\Lambda_W)_0(u,v) = - (\Lambda_W)_0([u,v]) + \Lie_{\rho(u)}((\Lambda_W)_0(v)) -\Lie_{\rho(v)}((\Lambda_W)_0(u)),\label{eq:LW0}\\
& (d^h\Lambda_W)_1(u)(v) = \Lie_{\rho(u)}((\Lambda_W)_1(v))-(\Lambda_W)_1([u,v])+ i_{\rho(v)}((\Lambda_W)_0(u)),\label{eq:LW1}\\
& (d^h\Lambda_W)_2(u)= -i_{\rho(u)}((\Lambda_W)_1(u)),\label{eq:LW2}
\end{align}
for $u,v \in \Gamma(A)$.

We also need the expression for $d^v$ in the following particular
situation. Any bundle map $\mu: A \to \wedge^k T^*M$ (equivalently
seen as a $C^\infty(M)$-linear map $\mu:\Gamma(A)\rmap \Omega^k(M)$)
defines an element $\mu_W \in W^{1,k}(A)$ by
\begin{equation}\label{eq:muW}
(\mu_W)_0 =\mu, \;\;\; (\mu_W)_1=0.
\end{equation}
In this case, $d^v(\mu_W)\in W^{1,k+1}(A)$ is defined by (see
\cite[Sec.~3.1]{AC})
\begin{equation}\label{eq:dvmu}
(d^v\mu_W)_0(u)= -\dd \mu(u),\;\;\; (d^v\mu_W)_1=\mu.
\end{equation}

\begin{proposition}\label{prop:IMcocicle}
Consider the map $\psi: \Omega^k_{\mathrm{lin}}(A)\rmap W^{1,k}(A)$,
$k=1,2,\ldots$,
$$
\Lambda=\dd\Lambda_\mu + \Lambda_\nu \, \mapsto \, \Lambda_W :=
-d^v\mu_W + \nu_W.
$$
The following holds:
\begin{enumerate}
\item $\psi$ induces a $C^\infty(M)$-linear isomorphism
$\Omega^\bullet_{\mathrm{lin}}(A)\stackrel{\sim}{\rmap}W^{1,\bullet}(A)$.
\item $\psi \circ \dd = - d^v\circ \psi$.
\item $\psi$ restricts to a linear isomorphism $\Omega^k_{\mathrm{IM}}(A)\stackrel{\sim}{\rmap}
\Ker(d^h|_{W^{1,k}(A)})$.
\end{enumerate}
\end{proposition}

\begin{proof}
It is clear from \eqref{eq:muW} and \eqref{eq:dvmu} that the map
$\psi$ is injective. Let us check that any $\Lambda_W\in W^{1,k}(A)$
can be written in the form $- d^v\mu_W + \nu_W$ for
$C^\infty(M)$-linear maps $\mu: \Gamma(A)\rmap \Omega^{k-1}(M)$,
$\nu:\Gamma(A)\rmap \Omega^k(M)$. Let us write
$\Lambda_W=((\Lambda_W)_0,(\Lambda_W)_1)$, and set $\mu= -
(\Lambda_W)_1$. Then $(d^v\mu_W)_1 = - (\Lambda_W)_1$, so the
element $c = d^v\mu_W + \Lambda_W \in W^{1,k}(A)$ is such that
$c_1=0$, which implies that $c=\nu_W$ for a bundle map $\nu:A\rmap
\wedge^kT^*M$. The $C^\infty(M)$-linearity of $\psi$ results from
the following properties: $f\dd\Lambda_\mu = \dd\Lambda_{f\mu} -
\Lambda_{\dd f\wedge \mu}$ and $d^v(f\mu)_W=fd^v\mu_W - (\dd f\wedge
\mu)_W$. Hence $(1)$ is proven.

To prove $(2)$, writing $\Lambda = \dd\Lambda_\mu + \Lambda_\nu$, we
have $\dd\Lambda = \dd\Lambda_\nu$. By definition of $\psi$, it
follows that $\psi(\dd\Lambda)= - d^v\nu_W$. On the other hand,
$-d^v(\psi(\Lambda)) = -d^v(-d^v\mu_W + \nu_W)= - d^v\nu_W$, hence
$(2)$ holds.

For $(3)$, we must consider the condition $d^h\Lambda_W=0$. Written
in terms of its components \eqref{eq:LW0}, \eqref{eq:LW1}, and
\eqref{eq:LW2}, we obtain three equations involving $(\Lambda_W)_0$
and $(\Lambda_W)_1$, which must be shown to agree with conditions
\eqref{eq:IM1}, \eqref{eq:IM2} and \eqref{eq:IM3} in
Thm.~\ref{thm:mainresult}. Using \eqref{eq:muW}, \eqref{eq:dvmu}, we
see that
\begin{equation}\label{eq:LW01}
(\Lambda_W)_0(u) = \dd\mu(u) + \nu(u) ,\qquad (\Lambda_W)_1(u) =
-\mu(u)
\end{equation}
for all $u\in \Gamma(A)$, and it is clear that \eqref{eq:LW2} and
\eqref{eq:LW1} coincide with conditions \eqref{eq:IM1} and
\eqref{eq:IM2}, respectively.

For the degree-0 condition \eqref{eq:LW0}, using \eqref{eq:LW01} and
\eqref{eq:IM2}, we find
$$
\nu([u,v])= \dd i_{\rho(v)}d\mu(u) + \dd i_{\rho(v)}\nu(u)
+\Lie_{\rho(u)}\nu(v) -\Lie_{\rho(v)}\dd\mu(u) -
\Lie_{\rho(v)}\nu(u).
$$
Using Cartan's formula $\Lie_X = i_X\dd + \dd i_X$, one directly
verifies that the last equation agrees with \eqref{eq:IM3}.
\end{proof}

Let $\Grd$ be a source-simply-connected Lie groupoid over $M$, with
Lie algebroid $A\rmap M$. There is a double complex
$\Omega^k(\Grd^{(p)})$ associated to $\Grd$, known as the
\textit{Bott-Shulman complex}, see \cite{Bott}. It is equipped with
a differential $\partial: \Omega^k(\Grd^{(p)})\rmap
\Omega^k(\Grd^{(p+1)})$, as well as the de Rham differential $\dd:
\Omega^k(\Grd^{(p)})\rmap \Omega^{k+1}(\Grd^{(p)})$. The \textit{Van
Est isomorphism} constructed in \cite{AC} relates the cohomologies
of $\Omega^k(\Grd^{(p)})$ and $W^{p,k}(A)$. We will only need a few
results of the theory,  for $p=0,1$.

For $p=0$, $\Omega^k(\Grd^{(0)})=\Omega^k(M)=W^{0,k}(A)$, and
$$
\partial: \Omega^k(M)\rmap \Omega^k(\Grd), \qquad \partial(\eta)=
\tar^*\eta - \sour^*\eta.
$$
For $p=1$, the differential $\partial: \Omega^k(\Grd)\rmap
\Omega^k(\Grd^{(2)})$ is
$$
\partial (\alpha) = pr_1^*\alpha - m^*\alpha + pr_2^*\alpha,
$$
and the Van Est map of \cite[Sec.~4]{AC} restricts to a map
\begin{equation}\label{eq:VE}
\mathbb{V}:
\Omega^k_{\mathrm{mult}}(\Grd)=\ker(\partial|_{\Omega^k(\Grd)})
\rmap \ker(d^h|_{W^{1,k}(A)})\subset W^{1,k}(A),
\end{equation}
given by
\begin{equation}\label{eq:V01}
\mathbb{V}(\alpha)_0(u) = \epsilon^*(\dd i_u\alpha + i_u
\dd\alpha),\qquad \mathbb{V}(\alpha)_1(u) = -\epsilon^*(i_u\alpha)
\end{equation}
where $\alpha \in \Omega^k_{\mathrm{mult}}(\Grd)$, $u\in \Gamma(A)$
(we view $A$ as a subbundle of $T\Grd|_M$) and $\epsilon:M\to \Grd$
is the unit map of $\Grd$. The map $\mathbb{V}$ satisfies
\begin{equation}\label{eq:Vprop}
\mathbb{V} \circ \dd = - d^v \circ \mathbb{V},\;\;\;\;
\mathbb{V}(\partial(\eta))=d^h\eta,
\end{equation}
for $\eta \in \Omega^k(M)$. The general Van Est isomorphism of
\cite{AC} implies that the induced map
\begin{equation}\label{eq:Visom}
\frac{\Omega^k_{\mathrm{mult}}(\Grd)}{\mathrm{Im}(\partial|_{\Omega^k(M)})}{\rmap}
\frac{\ker(d^h|_{W^{1,k}(A)})}{\mathrm{Im}(d^h|_{\Omega^k(M)})}
\end{equation}
is a bijection. In this specific situation, a stronger fact holds.

\begin{proposition}\label{prop:bijec} The map $\mathbb{V}: \Omega^k_{\mathrm{mult}}(\Grd) \rmap
\ker(d^h|_{W^{1,k}(A)})$ is a bijection.
\end{proposition}

The proof of the proposition uses the following observation (cf.
\cite[Sec.~6]{AC}).

\begin{lemma}\label{lem:bijec}
Let $\sigma \in {W^{1,k}(A)}$, $\omega\in
\Omega^{k+1}_{\mathrm{mult}}(\Grd)$ be such that $d^h\sigma=0$ and
$\mathbb{V}(\omega)=-d^v\sigma$. Then there exists a unique $\beta
\in \Omega^{k}_{\mathrm{mult}}(\Grd)$ such that
$$
\mathbb{V}(\beta)=\sigma, \;\;\; \dd\beta=\omega.
$$
\end{lemma}

\begin{proof}
The key fact to prove the lemma is shown in \cite[Lem.~6.3]{AC}: for
a closed $k$-form $\alpha \in \Omega^k_{\mathrm{mult}}(\Grd)$,
$\mathbb{V}(\alpha)=0$ if and only if $\alpha=0$. As an application,
we see that $\omega$ is necessarily closed, since
$\mathbb{V}(\dd\omega) = d^v(d^v\sigma)=0$.

Since $d^h\sigma=0$, the isomorphism \eqref{eq:Visom} implies that
there exists $\widetilde{\beta} \in \Omega^k_{\mathrm{mult}}(\Grd)$
such that $\mathbb{V}(\widetilde{\beta})= \sigma + d^h\eta$. If
$\beta = \widetilde{\beta}-\partial\eta$, then by \eqref{eq:Vprop}
we have
$$
\mathbb{V}(\beta) =
\mathbb{V}(\widetilde{\beta})-\mathbb{V}(\partial\eta) = \sigma.
$$
To conclude that $\dd\beta=\omega$, note that $\dd\beta-\omega$ is
multiplicative, closed, and $\mathbb{V}(\dd\beta -\omega)=-d^v\sigma
+ d^v\sigma = 0$.
\end{proof}

We can now prove the proposition.

\begin{proof} (of Prop.~\ref{prop:bijec})
Let us fix $\xi \in {W^{1,k}(A)}$, $d^h\xi =0$. Let $\sigma =
-d^v\xi$. Then $d^v\sigma =0$, $d^h\sigma=0$, and
Lemma~\ref{lem:bijec} implies that there exists a unique $\beta\in
\Omega^{k+1}_{\mathrm{mult}}(\Grd)$ such that
\begin{equation}\label{eq:beta}
\mathbb{V}(\beta)=\sigma, \;\;\;\; \dd\beta=0.
\end{equation}
Since $\mathbb{V}(\beta)=-d^v\xi$, and, by assumption, $d^h\xi=0$,
we can apply Lemma~\ref{lem:bijec} to conclude that there exists a
unique $\theta \in \Omega^{k}_{\mathrm{mult}}(\Grd)$ such that
$\mathbb{V}(\theta)=\xi$ and $\dd\theta = \beta$. But notice that
the condition $\dd\theta = \beta$ is automatically satisfied if
$\mathbb{V}(\theta)=\xi$, since $\mathbb{V}(\dd\theta)=\sigma$ and
the conditions in \eqref{eq:beta} determine $\beta$ uniquely.
\end{proof}

Composing the bijection \eqref{eq:VE} with the identification
$\Omega^k_{\mathrm{IM}}(A)\cong \ker(d^h|_{W^{1,k}}(A))$ of $(3)$ in
Prop.~\ref{prop:IMcocicle}, we obtain a bijection
$$
\Omega^k_{\mathrm{mult}}(\Grd) \stackrel{\sim}{\rmap}
\Omega^k_{\mathrm{IM}}(A).
$$
Using \eqref{eq:LW01} and \eqref{eq:V01}, we see that this bijection
is explicitly given by $\alpha \mapsto (\mu,\nu)$, where $\mu, \nu$
are defined as in \eqref{eq:mu}, \eqref{eq:nu}, hence agreeing with
Theorem~\ref{thm:1-1}.

%%%%%%%%%%%%%%%%%%%%%%%%%%%%%%%%%%%%%%%%%%%%%%%%%%%%%%%%%%%%%%%%%%%%%%%%
%\appendix
\section{The dual picture: multiplicative multivector
fields}\label{sec:dual}

In this section, we illustrate how the techniques used in the paper
to study infinitesimal versions of multiplicative forms can be
equally applied to multiplicative multivector fields.

We keep the notation introduced in Section \ref{subsec:tancot}. We
focus on the cotangent bundle $c_A: T^*A\to A$ of a vector bundle
$A\to M$, described in local coordinates $(x^j,u^d,p_j,\xi_d)$,
where $(x^j,u^d)$ are relative to a basis of local sections
$\{e_d\}$ of $A$. The local coordinates on $A^*$ relative to the
dual basis $\{ e^d \}$ are denoted by $(x^j,\xi_d)$; recall from
\eqref{eq:r} that we have a vector-bundle structure $r: T^*A\to
A^*$, $(x^j,u^d,p_j,\xi_d)\mapsto (x^j,\xi_d)$. As in Section
\ref{subsec:struc}, we also consider the $k$-fold direct sum
$\oplus_A^kT^*A$, described by local coordinates
$(x^j,u^d,p_j^1,\ldots,p_j^k,\xi_d^1,\ldots,\xi_d^k)$, as a vector
bundle over $\oplus^k A^*$, with projection map
$(x^j,u^d,p_j^1,\ldots,p_j^k,\xi_d^1,\ldots,\xi_d^k)\mapsto
(x^j,\xi_d^1,\ldots,\xi_d^k)$.

As in Section~\ref{subsec:corelinear}, we will need special sections
of the bundle $\oplus^k_A T^*A\to \oplus^k A^*$. For the bundle
$T^*A\rmap A^*$, we consider local sections
\begin{equation}\label{eq:secT*A}
\dd\widehat{x}^i(x^j,\xi_d)=(x^j,0,\delta^i_j,\xi_d),\;\;\;
e_a^L(x^j,\xi_d)=(x^j,\delta_a^d,0,\xi_d),
\end{equation}
which are core and linear sections, respectively; these sections
generate the module of local sections of $T^*A\rmap A^*$, and the
projection $T^*A\rmap A$ maps core sections to the zero section of
$A\rmap M$ and linear sections $e_a^L$ to the section $e_a$. More
generally, local sections of $\oplus^k_A T^*A\to \oplus^k A^*$ are
generated by sections of types
\begin{align}
&\dd\widehat{x}^{i,n}(\xi^1\oplus\ldots\oplus\xi^k)= 0(\xi^1)\oplus
\ldots 0(\xi^{n-1})\oplus \dd\widehat{x}^i(\xi^n) \oplus
0(\xi^{n+1}) \ldots \oplus 0(\xi^k),
\label{eq:secT*A1}\\
& (e_a^{L})^k(\xi^1\oplus\ldots\oplus\xi^k) = e_a^L(\xi^1)\oplus
\ldots \oplus e_a^L(\xi^k),\label{eq:secT*A2}
\end{align}
where $\xi^1\oplus\ldots\oplus\xi^k\in \oplus^kA^*$ and $0:A^*\to
T^*A$, $0(x^j,\xi_d)=(x^j,0,0,\xi_d)$, is the zero section. For each
$k$, we will use these sections to express the natural Lie algebroid
structure on $\oplus^k_A T^*A\to \oplus^k A^*$, similarly to Section
\ref{subsec:Liealg}.

Using the notation in \eqref{eq:liestruc}, the defining relations
for the cotangent Lie algebroid structure on $T^*A\to A^*$ are
\begin{align}
&[\dd\widehat{x}^i,\dd\widehat{x}^j]_{\st{T^*A}}=0,\label{eq:T*A0} \\
& [e_a^L,\dd\widehat{x}^j]_{\st{T^*A}}= \frac{\partial
{\rho}_a^j}{\partial x^i}\dd\widehat{x}^i, \;\;
[e_a^L,e_b^L]_{\st{T^*A}}=-\frac{\partial{C}^{c}_{ab}}{\partial x^i}
\xi_c \dd\widehat{x}^i +
C_{ab}^c e_c^L,\label{eq:T*A1}\\
&
\rho_{\st{T^*A}}(\dd\widehat{x}^i)=\rho^{i}_{d}\frac{\partial}{\partial
\xi_d},\;\;\;
\rho_{\st{T^*A}}(e_a^L)=\rho^{j}_{a}\frac{\partial}{\partial {x}^j}
+ C_{ab}^c \xi_c \frac{\partial}{\partial \xi_b}.\label{eq:T*A2}
\end{align}
This Lie algebroid structure is extended to direct sums $\oplus^k_A
T^*A\to \oplus^k A^*$ in total analogy to what was done for the
tangent Lie algebroid in Section \ref{subsec:Liealg}; we adopt the
simplified notation $\rho_k = \rho_{\oplus^k_A T^*A}$ and
$[\cdot,\cdot]_k=[\cdot,\cdot]_{\oplus^k_A T^*A}$ for the resulting
anchor and bracket\footnote{Since tangent Lie algebroids are not
used in this section, this notation should not cause any confusion
with the one in Section \ref{subsec:Liealg}.}. Explicitly, the
anchor is given by
\begin{equation}\label{eq:anc*}
\rho_k(\dd\widehat{x}^{i,n})=\rho^i_d\frac{\partial}{\partial
\xi^n_d}, \qquad  \rho_k((e_a^L)^k)=
\rho_a^j\frac{\partial}{\partial x^j} +
C_{ab}^c\xi_c^n\frac{\partial}{\partial \xi_b^n},
\end{equation}
whereas for the bracket we have
\begin{align}
& [\dd\widehat{x}^{i,n},\dd\widehat{x}^{j,m}]_k=0,\;\; \;
[(e_a^L)^k,\dd\widehat{x}^{j,m}]_k=\frac{\partial\rho_a^j}{\partial
x^i} \dd\widehat{x}^{i,m} \label{eq:br*1}\\
& [(e_a^L)^k,(e_b^L)^k]_k = C_{ab}^d (e_d^L)^k - \frac{\partial
C_{ab}^c}{\partial x^i}\xi_c^n \dd\widehat{x}^{i,n}.\label{eq:br*2}
\end{align}

%%%%%%%%%%%%%%%%%%%%%%%%%%%%%%%%%%%%%%%%%%%%%%%%%%%%%%%%%%%%%%%%%%%%%%%%%%%
\subsection{Linear multivector fields and derivations}

Let $\pi \in \cX^k(A)=\Gamma(\wedge^k TA)$ be a $k$-vector field on
the total space of a vector bundle $q_A: A\to M$. Let us consider
the function (cf. \eqref{eq:lbar})
$$\overline{\pi}: \oplus^k_AT^*A \to \mathbb{R},\;\;\;
\overline{\pi}(\Upsilon_1,\ldots,\Upsilon_k)=i_{\Upsilon_k}\ldots
i_{\Upsilon_1}\pi.
$$
We say that $\pi\in \cX^k(A)$ is \textbf{linear} if $\overline{\pi}$
defines a vector-bundle map
\begin{equation}\label{eq:linepi}
\xymatrix{\oplus_A^{k} T^*A  \ar[r]^{\;\;\;\;\;\overline{\pi}} \ar[d] & \mathbb{R} \ar[d] \\
 \oplus^{k}A^* \ar[r]_{} &  \{*\}, }
\end{equation}
similarly to \eqref{eq:Lvbmap}. One can directly verify that the
notion of linear multivector field agrees with the one considered in
\cite[Section~3.2]{ILX}. The space of linear $k$-vector fields is
denoted by $\cX^k_{\mathrm{lin}}(A)$. As in Lemma~\ref{lem:linear}
(cf. \eqref{eq:linearlocal}), $\pi$ is expressed in local
coordinates $(x^j,u^d)$ of $A$ as
\begin{align}
\pi = & \frac{1}{k!}\pi^{b_1\ldots b_k}_d(x) u^d \frac{\partial}{\partial u^{b_1}}\wedge \ldots\wedge  \frac{\partial}{\partial u^{b_k}}
+ \label{eq:pilinear}\\
& \frac{1}{(k-1)!}\pi^{b_1\ldots b_{k-1} j}(x)
\frac{\partial}{\partial u^{b_1}}\wedge \ldots\wedge
\frac{\partial}{\partial u^{b_{k-1}}}\wedge \frac{\partial}{\partial
x^j}. \nonumber
\end{align}

We have the following analog of Proposition~\ref{prop:struc} for
linear multivector fields, proven in \cite[Prop.~ 3.7]{ILX}: there
is a 1-1 correspondence between elements in
$\cX^k_{\mathrm{lin}}(M)$ and pairs $(\delta_0,\delta_1)$, where
$\delta_0:C^\infty(M)\to \Gamma(\wedge^{k-1}A)$ and
$\delta_1:\Gamma(A)\to \Gamma(\wedge^kA)$ are linear maps satisfying
\begin{equation}\label{eq:deriv01}
\delta_0(fg)=g \delta_0 f + f \delta_0 g,\;\;\;\;
\delta_1(fu)=(\delta_0 f)\wedge u + f \delta_1 u,
\end{equation}
for all $f,g\in C^\infty(M)$ and $u\in \Gamma(A)$. Equivalently, one
may view such pairs $(\delta_0,\delta_1)$ as restrictions of linear
maps $\delta: \Gamma(\wedge^\bullet A)\to
\Gamma(\wedge^{\bullet+k-1}A)$ satisfying the property
\begin{equation}\label{eq:gers0}
\delta(u\wedge v) = (\delta u)\wedge v + (-1)^{p(k-1)}u\wedge
(\delta v)
\end{equation}
for $u\in \Gamma(\wedge^p A)$ and $v\in \Gamma(\wedge^q A)$; i.e.,
$\delta$ is a degree-$(k-1)$ derivation of the exterior algebra
$\Gamma(\wedge^\bullet A)$. For this reason, we denote both maps
$\delta_0$ and $\delta_1$ by $\delta$. The explicit correspondence
between $\pi$ and $\delta$ is given by\footnote{To see that
\eqref{eq:pidelta1} and \eqref{eq:pidelta2} determine the linear
$k$-vector $\pi$, note that fibres of $T^*A \to A$ are generated by
elements of types $\dd l_\xi$ and $\dd q_A^*f$, and by linearity
$i_{\dd q_A^*f_2}i_{\dd q_A^*f_1}\pi=0$. \label{foot:linear}} (see
\cite[Section~3.2]{ILX})
\begin{align}
&\pi(\dd l_{\xi^1},\dots,\dd l_{\xi^{k-1}},\dd q_A^* f) =    q_A^*
\SP{\delta
f,\xi^1\wedge\ldots\wedge \xi^{k-1}},\label{eq:pidelta1} \\
& \pi(\dd l_{\xi^1},\ldots,\dd l_{\xi^k})(u)=
\sum_{i=1}^k(-1)^{i+k}\pi(\dd l_{\xi^1},\ldots,\widehat{\dd
l_{\xi^{i}}},\ldots,\dd l_{\xi^k},\dd q_A^*\SP{\xi^i,u})
\label{eq:pidelta2}\\
& \qquad \qquad \qquad \qquad \;\;\;\;\; -\SP{\delta u,
\xi^1\wedge\ldots\wedge \xi^k},\nonumber
\end{align}
where $f\in C^\infty(M)$, $u\in \Gamma(A)$, $\xi^1,\ldots,\xi^k \in
\Gamma(A^*)$, and $l_{\xi^i}\in C^\infty(A)$ is the linear function
$l_{\xi^i}(u)=\SP{\xi^i,u}$. In coordinates, we have
\begin{equation}\label{eq:coord}
\delta x^i = \frac{1}{(k-1)!}\pi^{b_1\ldots b_{k-1}i}(x)
e_{b_1}\wedge\ldots\wedge e_{b_{k-1}},\;\;\; \delta e_a =
-\frac{1}{k!}\pi_a^{b_1\ldots b_k}(x) e_{b_1}\wedge \ldots\wedge
e_{b_k},
\end{equation}
where $\{e_d\}$ is a basis of local sections of $A$.

Let $A\to M$ be a Lie algebroid. The Lie bracket $[\cdot,\cdot]$ on
$\Gamma(A)$ has a natural extension (still denoted by
$[\cdot,\cdot]$) to the exterior algebra $\Gamma(\wedge^\bullet A)$,
$$
[\cdot,\cdot]: \Gamma(\wedge^p A)\times \Gamma(\wedge^q A)\to
\Gamma(\wedge^{p+q-1}A),
$$
making it into a \textit{Gerstenhaber
algebra} (see e.g. \cite{CW}): for $u\in \Gamma(\wedge^pA)$, $v\in
\Gamma(\wedge^qA)$, and $w\in \Gamma(\wedge^rA)$, we have
\begin{align}
&[u,v]=-(-1)^{(p-1)(q-1)}[v,u], \label{eq:gers1}\\
&[u, v\wedge w] = [u,v]\wedge w + (-1)^{(p-1)q}v\wedge[u,w].
\label{eq:gers2}
\end{align}
The next result is the analog of Theorem~\ref{thm:mainresult} for
linear multivector fields.

\begin{theorem}\label{thm:maindual}
Let $\pi \in \cX^k_{\mathrm{lin}}(A)$ be a linear $k$-vector field
on a Lie algebroid $A$, and let $\delta:\Gamma(\wedge^\bullet A)\to
\Gamma(\wedge^{\bullet + k-1}A)$ be the associated derivation (as in
\eqref{eq:pidelta1} and \eqref{eq:pidelta2}). Then the map
$\overline{\pi}$ \eqref{eq:linepi} is a Lie algebroid morphism if
and only if
\begin{equation}\label{eq:gers}
\delta [u,v]=[\delta u, v] + (-1)^{(p-1)(k-1)}[u,\delta v],
\end{equation}
for all $u\in
\Gamma(\wedge^p A)$, $v\in \Gamma(\wedge^q A)$ (i.e.,  $\delta$ is a
$(k-1)$-derivation of the Gerstenhaber bracket).
\end{theorem}

To draw a clear parallel with Theorem~\ref{thm:mainresult}, we
denote by $\cX^{k}_{\mathrm{IM}}(A)$ the space of degree $(k-1)$
derivations $\delta:\Gamma(\wedge^\bullet A)\to
\Gamma(\wedge^{\bullet + k-1}A)$ of the Gerstenhaber structure
(i.e., \eqref{eq:gers0} and \eqref{eq:gers} hold), in analogy with
IM $k$-forms.

\begin{proof}
We work locally, so the condition that $\overline{\pi}$ is a Lie
algebroid morphism is
\begin{equation}\label{eq:liealgcond2}
\overline{\pi}([\Upsilon_1,\Upsilon_2]_k)=\Lie_{\rho_k(\Upsilon_1)}\overline{\pi}(\Upsilon_2)
- \Lie_{\rho_k(\Upsilon_2)}\overline{\pi}(\Upsilon_1),
\end{equation}
where $\Upsilon_1, \Upsilon_2$ are local sections of
$\oplus^k_AT^*A\to \oplus^kA^*$ of types \eqref{eq:secT*A1} or
\eqref{eq:secT*A2} (cf. \eqref{eq:liealgcond}); hence, just as in
the proof of Theorem~\ref{thm:mainresult}, there are 3 cases to be
analyzed. The assertion of Theorem~\ref{thm:maindual} is a direct
consequence of the following claims:

\begin{itemize}
\item[(c1)] Let $\Upsilon_1 = \dd\widehat{x}^{i,l}$ and $\Upsilon_2 =
\dd\widehat{x}^{j,m}$. If $l=m$, then \eqref{eq:liealgcond2} is
automatically satisfied; if $l\neq m$, then \eqref{eq:liealgcond2}
is equivalent to
\begin{equation}\label{eq:r1}
\delta [x^i,x^j] = [\delta x^i,x^j] + (-1)^{k-1}[x^i,\delta x^j] =0.
\end{equation}

\item[(c2)] Let $\Upsilon_1 = \dd\widehat{x}^{i,l}$ and $\Upsilon_2 =
(e_b^L)^k$. Then \eqref{eq:liealgcond2} is equivalent to
\begin{equation}\label{eq:r2}
\delta [x^i,e_b] = [\delta x^i,e_b] + (-1)^{k-1}[x^i,\delta e_b].
\end{equation}

\item[(c3)] Let $\Upsilon_1=(e_a^L)^k$ and $\Upsilon_2 =
(e_b^L)^k$. Then \eqref{eq:liealgcond2} is equivalent to
\begin{equation}\label{eq:r3}
\delta [e_a,e_b] = [\delta e_a,e_b] + [e_a,\delta e_b].
\end{equation}

\end{itemize}

In order to prove claims (c1), (c2) and (c3), we need some general
observations. For any function $F: \oplus^k A^*\to \mathbb{R}$ which
is $k$-linear over $C^\infty(M)$ and skew symmetric, let $\Phi_F\in
\Gamma(\wedge^kA)$ be the unique element such that
$$
F(\xi^1,\ldots,\xi^k) = \SP{\Phi_F, \xi^1\wedge \ldots\wedge \xi^k}.
$$
E.g., for $F(\xi^1,\ldots,\xi^k)=F^{b_1\ldots b_k}\xi^1_{b_1}\ldots
\xi^k_{b_k}$ with $F^{b_1\ldots b_k}$ totally antisymetric in its indices, we have $\Phi_F = \frac{1}{k!} F^{b_1\ldots
b_k} e_{b_1}\wedge \ldots \wedge e_{b_k}$. We will
consider the cases where $F = \overline{\pi}(\dd\widehat{x}^{j,m})$
and $F= \overline{\pi}((e_a^L)^k)$. Using the local expressions
\eqref{eq:secT*A1}, \eqref{eq:secT*A2} as well as
\eqref{eq:pilinear} and \eqref{eq:coord}, one may directly verify
the following identities:
\begin{align}
& \overline{\pi}(\dd\widehat{x}^{j,m}(\xi^1,\ldots,\xi^k)) =
(-1)^{k-m}\SP{\delta x^j, \xi^1\wedge \ldots \wedge
\widehat{\xi^{m}}\wedge \ldots \wedge \xi^k}, \label{eq:Id1}
\\
& \overline{\pi}((e_a^L)^k(\xi^1,\ldots,\xi^k)) = -\SP{\delta e_a,
\xi^1\wedge \ldots \wedge \xi^k},\label{eq:Id2}
\end{align}
where the notation $\xi^1\wedge \ldots \wedge
\widehat{\xi^{m}}\wedge \ldots \wedge \xi^k$ means that $\xi^m$ is
omitted.

Let us now consider $F(\xi^1,\ldots,\xi^k)=F^{b_1\ldots
b_k}\xi^1_{b_1}\ldots \xi^k_{b_k}$ and the vector fields
$\rho_k(\dd\widehat{x}^{i,l})$ and $\rho_k((e_a^L)^k)$ on $\oplus^k
A^*$, see \eqref{eq:anc*}. Then a direct computation shows the
following identities:
\begin{align}
& \Lie_{\rho_k(\dd\widehat{x}^{i,l})} (F (\xi^1,\ldots,\xi^k)) =
(-1)^l\SP{[x^i,\Phi_F], \xi^1\wedge \ldots \wedge
\widehat{\xi^{l}}\wedge \ldots \wedge \xi^k},\label{eq:Id3}
\\
&  \Lie_{\rho_k((e_a^L)^k)}(F (\xi^1,\ldots,\xi^k)) =
\SP{[e_a,\Phi_F], \xi^1\wedge \ldots \wedge \ldots \wedge \xi^k}.
\label{eq:Id4}
\end{align}
From \eqref{eq:Id1} and \eqref{eq:Id3}, we directly see, assuming
that $l< m$, that
\begin{align*}
&\Lie_{\rho_k(\dd\widehat{x}^{i,l})}
(\overline{\pi}(\dd\widehat{x}^{j,m})(\xi^1,\ldots,\xi^k)) =
\\&\qquad \qquad \qquad
(-1)^{l+k-m}\SP{[x^i,\delta x^j],
\xi^1\wedge\ldots\wedge\widehat{\xi^l}\wedge\ldots\wedge\widehat{\xi^m}\wedge\ldots,\xi^k},
\\
&\Lie_{\rho_k(\dd\widehat{x}^{j,m})}
(\overline{\pi}(\dd\widehat{x}^{i,l})(\xi^1,\ldots,\xi^k)) =
\\&\qquad \qquad  \qquad
(-1)^{m-1+k-l}\SP{[x^j,\delta x^i],
\xi^1\wedge\ldots\wedge\widehat{\xi^l}\wedge\ldots\wedge\widehat{\xi^m}\wedge\ldots\wedge\xi^k}.
\end{align*}
Combining these two equations with \eqref{eq:gers1}, we conclude
that claim (c1) holds.

To prove the other two claims, note first that the derivation
property for functions in \eqref{eq:deriv01} implies that
\begin{equation}\label{eq:deltaf}
\delta f = \frac{\partial f}{\partial x^j}\delta x^j,\;\;\;\; f\in
C^\infty(M).
\end{equation}
As a result, since $[e_b,x^i]=\Lie_{\rho(e_b)}x^i = \rho_b^i$, we
have that
\begin{equation}\label{eq:r21}
\delta [x^i,e_b] = -\delta [e_b,x^i] = - \frac{\partial
\rho_b^i}{\partial x^j}\delta x^j.
\end{equation}
Using the second formula in \eqref{eq:br*1} together with
\eqref{eq:Id1} and \eqref{eq:r21}, we obtain
\begin{align}\label{eq:r22}
\overline{\pi}([\dd\widehat{x}^{i,l},(e^L_b)^k]_k(\xi^1,\ldots,\xi^k))
&=
 -(-1)^{k-l}
\SP{\frac{\partial \rho_b^i}{\partial x^j}\delta x^j,\xi^1\wedge\ldots\wedge \widehat{\xi^l}\wedge\ldots\wedge \xi^k}\\
 & = (-1)^{k-l}\SP{\delta [x^i,e_b], \xi^1\wedge\ldots\wedge \widehat{\xi^l}\wedge\ldots\wedge \xi^k}.           \nonumber
\end{align}
Combining \eqref{eq:Id1} and \eqref{eq:Id4}, as well as
\eqref{eq:Id2} and \eqref{eq:Id3}, we immediately get
\begin{align}
&
\Lie_{\rho_k((e_b^L)^k)}(\overline{\pi}(\dd\widehat{x}^{i,l}(\xi^1,\ldots,\xi^k)))
=
(-1)^{k-l}\SP{[e_b,\delta x^i],\xi^1\wedge\ldots\wedge \widehat{\xi^l}\wedge\ldots\wedge \xi^k},  \label{eq:r23} \\
&
\Lie_{\rho_k(\dd\widehat{x}^{i,l})}(\overline{\pi}((e_b^L)^k(\xi^1,\ldots,\xi^k)))
 = -(-1)^l\SP{[x^i,\delta e_b], \xi^1\wedge\ldots\wedge \widehat{\xi^l}\wedge\ldots\wedge
 \xi^k}.\label{eq:r24}
\end{align}
Now claim (c2) is a direct consequence of \eqref{eq:r22},
\eqref{eq:r23} and \eqref{eq:r24}.

Finally, to prove (c3), we observe a few facts. From
\eqref{eq:deltaf}, we see that
\begin{equation}\label{eq:r31}
\delta[e_a,e_b] = \delta(C_{ab}^c e_c) = \frac{\partial
C_{ab}^c}{\partial x^j}\delta x^j \wedge e_c + C_{ab}^c\delta e_c.
\end{equation}
The usual formula for the wedge product gives us the identity
$$
\frac{\partial C_{ab}^c}{\partial x^j}\SP{\delta x^j \wedge e_c,
\xi^1\wedge\ldots\wedge \xi^k} = \sum_{n=1}^k
(-1)^{k-n}\frac{\partial C_{ab}^c}{\partial x^j}\xi^n_c \SP{\delta
x^j,\xi^1\wedge\ldots\wedge \widehat{\xi
^n}\wedge\ldots\wedge\xi^k};
$$
using it, we immediately obtain from \eqref{eq:br*2}, \eqref{eq:Id1}
and \eqref{eq:Id2} that
\begin{equation}\label{eq:r32}
\overline{\pi}([(e_a^L)^k,(e_b^L)^k]_k(\xi^1,\ldots,\xi^k))=-\SP{\delta[e_a,e_b],\xi^1\wedge\ldots\wedge\xi^k}.
\end{equation}
On the other hand, from \eqref{eq:Id2} and \eqref{eq:Id4} we have
that
\begin{equation}\label{eq:r33}
\Lie_{\rho_k((e_a^L)^k)}(\overline{\pi}((e_b^L)^k(\xi^1,\ldots,\xi^k)))=-\SP{[e_a,\delta
e_b],\xi^1\wedge\ldots\wedge\xi^k}.
\end{equation}
Using \eqref{eq:r32} and \eqref{eq:r33}, we can immediately verify
that claim (c3) holds.
\end{proof}

%%%%%%%%%%%%%%%%%%%%%%%%%%%%%%%%%%%%%%%%%%%%%%%%%%%%%%%%%%%%%%%%%%%%%%%%%%%%%%%%%%%%%%%
\subsection{Infinitesimal description of multiplicative multivector fields}

We now discuss the analogs of the results in Section \ref{sec:mult}
for multiplicative multivector fields.

Let $\Grd$ be a Lie groupoid over $M$. Its cotangent bundle
$T^*\Grd$ has a natural Lie groupoid structure over $A^*$, known as
the \textbf{cotangent groupoid} of $\Grd$, see \cite{CDW} and
\cite{Mac-book} for a full description. For us, it will suffice to
recall that the unit map $\widetilde{\epsilon}:A^*\to T^*\Grd|_M$
identifies $A^*$ with the annihilator of $TM\subset T\Grd$, and that
the source map $\widetilde{\sour}: T^*\Grd \to A^*$ is defined by
\begin{equation}\label{eq:sourcet}
\SP{\widetilde{\sour}(\alpha_g),u} = \SP{\alpha_g,Tl_g(u-T\tar(u))},
\qquad  \alpha_g\in T_g^*\Grd, \; u\in A_{\sour(g)},
\end{equation}
where $l_g$ denotes left translation in $\Grd$. Note that
$\widetilde{\sour}$ is a vector-bundle map covering $\sour :\Grd\to
M$; using coordinates $(z^l)$ on $\Grd$, it has the form
\begin{equation}\label{eq:sexp}
\widetilde{\sour}(z^l,\alpha_l)= (\sour(z)^j,C^l_d(z)\alpha_l) \in
A^*|_{\sour(z)}.
\end{equation}
We will not need the explicit expression for $C^l_d(z)$, just to
note that $\widetilde{\sour}(\dd\tar^*f)=0$ for all $f\in
C^\infty(M)$ (by \eqref{eq:sourcet}), which implies that
\begin{equation}\label{eq:Cprop}
C^l_d(z)\frac{\partial (\tar^*f)}{\partial z^l}=0, \;\;\; \forall
f\in C^\infty(M).
\end{equation}

Similarly to what happens for the tangent groupoid, the cotangent
groupoid structure extends to direct sums $\oplus^k_\Grd T^*\Grd$
over $\oplus^k A^*$. A multivector field $\Pi \in \cX^k(\Grd)$ is
called \textbf{multiplicative} if the associated map
\begin{equation}\label{eq:Pi}
\overline{\Pi}: \oplus^k_\Grd T^*\Grd \to \mathbb{R},\;\;\;
\overline{\Pi}(\zeta_1,\ldots,\zeta_k)= i_{\zeta_k}\ldots
i_{\zeta_1} \Pi
\end{equation}
is a groupoid morphism (cf. Lemma~\ref{lem:multip}). We denote the
space of multiplicative $k$-vector fields on $\Grd$ by
$\cX^k_{\mathrm{mult}}(\Grd)$.

\begin{remark}\label{rem:sharp1}
We may equivalently consider the map
\begin{equation}\label{eq:sharp}
\Pi^\sharp: \oplus^{k-1}_\Grd T^* \Grd \to T\Grd,\;\;
{\Pi^\sharp}(\zeta_1,\ldots,\zeta_{k-1})= i_{\zeta_{k-1}}\ldots
i_{\zeta_1} \Pi,
\end{equation}
and verify that $\Pi$ is multiplicative if and only if $\Pi^\sharp$
is a groupoid morphism.
\end{remark}

Let us recall, see e.g. \cite{Mac-Xu}, the identification of Lie
algebroids
\begin{equation}\label{eq:thetaG}
\theta_\Grd : A(T^*\Grd) \stackrel{\sim}{\longrightarrow}
T^*(A\Grd),
\end{equation}
which extends to an identification $\theta_\Grd^k: A(\oplus^k_\Grd
T^*\Grd)=\prod_{\LF(c_\Grd)}A(T^*\Grd) \to \oplus_A^k T^*A$
($c_\Grd: T^*\Grd\to \Grd$ is a groupoid morphism). Given $\Pi \in
\cX^k_{\mathrm{mult}}(\Grd)$, we consider the infinitesimal map
$\LF(\overline{\Pi}):A(\oplus^k_\Grd T^*\Grd)\to \mathbb{R}$ (see
\eqref{eq:liemorphism}), as well as the composition
\begin{equation}\label{eq:liePi}
\LF(\overline{\Pi}) \circ (\theta_\Grd^k)^{-1}: \oplus_A^k T^*A \to
\mathbb{R}.
\end{equation}
The exact same arguments as in Lemma~\ref{lem:int} directly show
that there is a unique $k$-vector field $\LF(\Pi) \in \cX^k(A)$
satisfying
\begin{equation}\label{eq:liePi2}
\overline{\LF(\Pi)} = \LF(\overline{\Pi}) \circ
(\theta_\Grd^k)^{-1};
\end{equation}
moreover, the map
$$
\cX^k_{\mathrm{mult}}(\Grd) \longrightarrow \cX^k(A),\;\;\; \Pi
\mapsto \LF(\Pi),
$$
is a bijection onto the subspace of $k$-vector fields $\pi \in
\cX^k_{\mathrm{lin}}(A)$ for which $\overline{\pi}:\oplus_A^k T^*A
\to \mathbb{R}$ is a morphism of Lie algebroids. An immediate
consequence of Theorem~\ref{thm:maindual} is
\begin{corollary}\label{cor:1-1}
There is a bijective correspondence
\begin{equation}\label{eq:bijder}
\cX_{\mathrm{mult}}^k(\Grd)\longrightarrow
\cX^k_{\mathrm{IM}}(A),\;\;\; \Pi \mapsto \delta,
\end{equation}
where $\delta$ is the derivation associated with $\pi = \LF(\Pi) \in
\cX^k_{\mathrm{lin}}(A)$ (via \eqref{eq:pidelta1} and
\eqref{eq:pidelta2}).
\end{corollary}
This result is parallel to Theorem~\ref{thm:1-1}, except that it
provides no explicit way of computing $\delta$ directly out of $\Pi$
(analogous to \eqref{eq:mu1} and \eqref{eq:nu1}). This missing
aspect will be clarified in the next section.

\subsection{The universal lifiting theorem revisited}

For $u \in \Gamma(\wedge^p A)$, let us denote by $u^r$ the
corresponding right-invariant $p$-vector field on $\Grd$. As
observed in \cite[Section~2]{ILX}, given $\Pi \in
\cX_{\mathrm{mult}}^k(\Grd)$, then $[\Pi, u^r]$ is again right
invariant, which means that there exists $\delta_\Pi u \in
\Gamma(\wedge^{p+k-1}A)$ such that $(\delta_\Pi u)^r = [\Pi, u^r]$.
One can check that the map $\delta_\Pi : \Gamma(\wedge^p A)\to
\Gamma(\wedge^{p+k-1}A)$ is a derivation of the Gerstenhaber
structure, i.e., $\delta_\Pi \in \cX^k_{\mathrm{IM}}(M)$.

\begin{proposition}\label{prop:der}
The map $\cX_{\mathrm{mult}}^k(\Grd)\longrightarrow
\cX^k_{\mathrm{IM}}(A)$, $\Pi \mapsto \delta_\Pi$, where
$\delta_\Pi$ is defined by
\begin{equation}\label{eq:deltapi}
(\delta_\Pi f)^r = [\Pi, \tar^* f],\;\;\;(\delta_\Pi u)^r=[\Pi,u^r],
\end{equation}
for $f\in C^\infty(M)$ and $u\in \Gamma(A)$, coincides with the map
\eqref{eq:bijder}; in particular, it is a bijection.
\end{proposition}

The fact that the correspondence in Proposition~\ref{prop:der} is a
bijection is the \textit{universal lifting theorem} of \cite{ILX}
(see Theorem~2.34 therein),  which we recover here as a consequence
of Corollary~\ref{cor:1-1}. We need to collect some observations
before getting into the proof of Proposition~\ref{prop:der}.

Let us consider the isomorphism
\begin{equation}\label{eq:Theta}
\Theta: T(T^*\Grd)\longrightarrow T^*(T\Grd),\;\;
(z^j,\alpha_j,\dot{z}^j,\dot{\alpha}_j)\mapsto
(z^j,\dot{z}^j,\dot{\alpha}_j,\alpha_j),
\end{equation}
which is related to the identification $\theta_\Grd$ in
\eqref{eq:thetaG} via
\begin{equation}\label{eq:reltheta}
\theta_\Grd = (T\iota_A)^t \circ \Theta \circ \iota_{A(T^*\Grd)},
\end{equation}
where $(T\iota_A)^t$ is the fibrewise dual to the vector-bundle map
$T\iota_A: TA\to \iota_A^* T(T\Grd)$ (the composition in
\eqref{eq:reltheta} is well defined since $\Theta\circ
\iota_{A(T^*\Grd)}(A(T^*\Grd))\subset \iota_A^* T^*(T\Grd)$; this
can be derived directly from \eqref{eq:sexp}). For a $k$-vector
field $\Pi\in \cX^k(\Grd)$, its \textbf{tangent lift} is the
$k$-vector field $\Pi_T\in \cX^k(T\Grd)$ defined by the condition
(cf. \eqref{eq:alphaT})
\begin{equation}\label{eq:PiT1}
\overline{\Pi_T}= \dd \overline{\Pi}\circ (\Theta^{-1})^k,
\end{equation}
where $\dd \overline{\Pi}: T(\oplus_\Grd^k
T^*\Grd)=\prod^k_{Tc_\Grd}T(T^*\Grd) \to \mathbb{R}$ is the
differential of the function $\overline{\Pi}$ in $C^\infty (\oplus^k_\Grd T^*\Grd)$ defined by \eqref{eq:Pi}.

%If $\Pi$ is multiplicative, then $\Pi_T$ is multiplicative on the tangent groupoid (this follows from $\Theta$
%preserving groupoid structures, see e.g. \cite[Section~11]{Mac-book}).

\begin{remark}\label{rem:sharp2}
As observed in \cite{GU}, one may alternatively define the tangent
lift $\Pi_T$ in terms of $\Pi^\sharp$ \eqref{eq:sharp}:
\begin{equation}\label{eq:tanglift2}
\Pi_T^\sharp: \oplus_{T\Grd}^{k-1}T^*(T\Grd)
 \to T(T\Grd), \;\; \Pi_T^\sharp
= (J_\Grd)^{-1}\circ T\Pi^\sharp \circ (\Theta^{-1})^{(k-1)},
\end{equation}
where $J_\Grd: T(T\Grd)\to T(T\Grd)$ is the involution \eqref{eq:J}.
\end{remark}

When $\Pi$ is multiplicative, it follows from \eqref{eq:liePi2} that
\begin{equation}\label{eq:pitheta}
\overline{\pi} = \overline{\LF(\Pi)} = \overline{\Pi_T}\circ
(\Theta\circ \iota_{AT^*\Grd}\circ \theta_\Grd^{-1})^k.
\end{equation}
We will need this characterization of $\pi$ in the proof of
Proposition \ref{prop:der}.

For local computations, it will be convenient to consider adapted
local coordinates
\begin{equation}\label{eq:adapt}
(x^j,y^d) \; \mbox{ on $\Grd$ around $M\subset \Grd$,}
\end{equation}
where $y^d$ are coordinates along the $\sour$-fibres. We will also
use the induced coordinates $((x^j,y^d),(\dot{x}^j,\dot{y}^d))$ on
$T\Grd$, and similarly for $T^*\Grd$, $T(T^*\Grd)$ and $T^*(T\Grd)$.
In these coordinates, $\iota_A:A\to T\Grd|_M$,
$\iota_A(x^j,u^d)=((x^j,0),(0,u^d))$, and $T\iota_A: TA\to
\iota_A^*T(T\Grd)$ is given by
$$
T\iota_A\left(\dot{x}^j\frac{\partial}{\partial x^j} +
\dot{u}^d\frac{\partial}{\partial u^d}\right)\Big|_u =
\dot{x}^j\frac{\partial}{\partial x^j} +
\dot{u}^d\frac{\partial}{\partial \dot{y}^d}
\Big|_{\iota_A(u)},\;\;\; u\in A,
$$
whereas for $(T\iota_A)^t : \iota_A^* T^*(T\Grd)\to T^*A$ we have
$$
(T\iota_A)^t (p_jdx^j +\gamma_ady^a + \overline{p}_jd\dot{x}^j
+\overline{\gamma}_ad\dot{y}^a)|_{\iota_A(u)} = (p_jdx^j +
\overline{\gamma}_a du^a)|_u, \;\; u\in A.
$$
Since the unit map $\widetilde{\epsilon}:A^*\to T^*\Grd|_M$
identifies $A^*$ with the annihilator of $TM\subset T\Grd$, given
$\xi\in\Gamma(A^*)$, locally written as $(x^j,\xi_d)$, the local
1-form on $\Grd$ given by
\begin{equation}\label{eq:xit}
\widetilde{\xi}(x^j,y^d) = \xi_d(x) dy^d
\end{equation}
extends $\widetilde{\epsilon}(\xi(x))$ to a neighborhood of $M$ in
$\Grd$. We denote by $l_{\widetilde{\xi}}\in C^\infty(T\Grd)$ the
linear function determined by $\widetilde{\xi}$. The following lemma
is key to compare the map in Proposition~\ref{prop:der} with the map
\eqref{eq:bijder}.

\begin{lemma}\label{lem:J} Let $\mathcal{J} = \Theta\circ \iota_{AT^*\Grd}\circ \theta_\Grd^{-1}: T^*A\to
\iota_A^* T^*(T\Grd)$, and let $u_0\in A$. Then, for any $f\in
C^\infty(M)$ and $\xi\in \Gamma(A^*)$, we have
\begin{align}
& \mathcal{J}(\dd q_A^*f |_{u_0}) = \dd (\tar^* f)^{\vl} |_{\iota_A(u_0)} ,\label{eq:Ja}\\
& \mathcal{J}(\dd l_\xi|_{u_0}) = (\dd l_{\widetilde{\xi}} + \dd
h^{\vl})|_{\iota_A(u_0)},\label{eq:Jb}
\end{align}
where $^{\vl}$ means the pull-back of functions on $\Grd$ by
$p_\Grd:T\Grd\to \Grd$, and $h\in C^\infty(\Grd)$ is a function that
vanishes on $M\subset \Grd$.

\end{lemma}

\begin{proof} The proof follows from some observations, all of which
can be checked through computations in adapted local coordinates
$(x^j,y^d)$ as in \eqref{eq:adapt}.

The first observation one can directly verify is that
\begin{equation}\label{eq:J1}
(T\iota_A)^t (\dd(\tar^* f)^{\vl} |_{\iota_A(u)})=\dd
q_A^*f|_{u},\;\; u\in A.
\end{equation}

Using the local expression \eqref{eq:sexp} for the source map
$\widetilde{\sour}$, the property \eqref{eq:Cprop}, and the
definition of $\Theta$, a direct computation shows that
\begin{equation}\label{eq:Tst}
T\widetilde{\sour}(\Theta^{-1}(\dd (\tar^* f)^{\vl} |_{\iota_A(u)}))
= 0 \in TA^*|_{q_A(u)},\;\;\; q_A(u)\in M\subset A^*.
\end{equation}
It follows that $\Theta^{-1}(\dd(\tar^* f)^{\vl} |_{\iota_A(u)})$ is
in the image of $\iota_{A(T^*\Grd)}$, hence there is a unique
$\Upsilon \in T^*A|_{q_A(u)}$ such that
$$
\Theta^{-1}(\dd(\tar^* f)^{\vl}
|_{\iota_A(u)})=\iota_{A(T^*\Grd)}(\theta_\Grd^{-1}(\Upsilon)), \;
\mbox{ i.e., } \; \mathcal{J}(\Upsilon)= \dd(\tar^* f)^{\vl}
|_{\iota_A(u)}.
$$
Using \eqref{eq:J1} and \eqref{eq:reltheta}. we conclude that
$\Upsilon = \dd q_A^*f|_u$, which proves \eqref{eq:Ja}.

With respect to the coordinates $((x^j,y^d),(\dot{x}^j,\dot{y}^d))$
on $T\Grd$, one can write
\begin{equation}\label{eq:dlt}
\dd l_{\widetilde{\xi}}|_{\iota_A(u)} = \frac{\partial
\xi_a}{\partial x^i}u^a\dd x^i + \xi_a \dd\dot{y}^a \in
T^*(T\Grd)|_{\iota_A(u)},
\end{equation}
from where we conclude that
\begin{equation}\label{eq:leq}
(T\iota_A)^t (\dd l_{\widetilde{\xi}}|_{\iota_A(u)}) = \left
(\frac{\partial \xi_a}{\partial x^j}u^a\dd x^j + \xi_a \dd u^a
\right)\Big|_{u} = q_A^* \dd l_\xi |_{u} \;\in T^*A|_u.
\end{equation}
Let us now consider $\mathcal{J}(q_A^*\dd l_\xi) \in \iota_A^*
T^*(T\Grd)$. Since $\Theta^{-1}(\mathcal{J}(q_A^*\dd l_\xi))$ lies
in $T(T^*\Grd)|_{A^*}$, one can directly verify that
$\mathcal{J}(q_A^*\dd l_\xi)$ can be written as
$$
p_j\dd x^j + \gamma_a\dd y^a + \overline{\gamma}_a\dd \dot{y}^a,
$$
i.e., its components relative to $\dd\dot{x}^j$ vanish. By
\eqref{eq:reltheta}, $(T\iota_A)^t(\mathcal{J}(q_A^*\dd
l_\xi)|_u)=q_A^*\dd l_\xi|_u$, so from the second equality in
\eqref{eq:leq} we conclude that $p_j = \frac{\partial
\xi_a}{\partial x^j}u^a$ and $\overline{\gamma}_a=\xi_a$, i.e.,
$$
\mathcal{J}(q_A^*\dd l_\xi|_u) = \left( \frac{\partial
\xi_a}{\partial x^j}u^a \dd x^j + \overline{\gamma}_a \dd y^a +
{\xi}_a \dd\dot{y}^a\right)\Big |_{\iota_A(u)}.
$$
For each given $u_0\in A$, one can find $h\in C^\infty(\Grd)$
vanishing on $M\subset \Grd$ and such that $\dd
h^{\vl}|_{\iota_A(u_0)} = \overline{\gamma}_a(\iota_A(u_0))\dd y^a$,
and \eqref{eq:Jb} follows by a direct comparison with
\eqref{eq:dlt}.
\end{proof}

We will need the following immediate observations about linear
functions on vector bundles.

\begin{lemma}\label{lem:linf}
Let $q_B: B\to N$ be a vector bundle, with coordinates $(x^j,b^d)$
relative to a basis of local sections $\{e_d\}$, and consider
$b=b^de_d\in \Gamma(B)$, $b^{\vl}=b^d\frac{\partial}{\partial
b^d}\in \cX(B)$, and $\beta=\beta_d e^d \in \Gamma(B^*)$. Let
$l_\beta \in C^\infty(B)$ be the linear function defined by $\beta$,
and fix $b_0 = b(x_0) \in B$, for a given $x_0\in N$. Then
$\Lie_{b^{\vl}}l_\beta = q_B^* \SP{\beta,b}$ and
\begin{equation}\label{eq:linf}
l_\beta(b_0) = (\Lie_{b^{\vl}}l_\beta)(b_0).
\end{equation}
\end{lemma}

We now prove Proposition~\ref{prop:der}.

\begin{proof}(of Proposition~\ref{prop:der})

Let $\pi=\LF(\Pi)$, and consider $\xi^1,\ldots,\xi^{k-1}\in
\Gamma(A^*)$ and $f\in C^\infty(M)$. Let us fix $u_0 \in A$,
$x_0=q_A(u_0)\in M$. By \eqref{eq:pitheta} and Lemma~\ref{lem:J}, we
have
\begin{align}
\pi(\dd l_{\xi^1},\ldots,\dd l_{\xi{k-1}},dq_A^*f)|_{u_0} &=
\Pi_T(\mathcal{J}\dd l_{\xi^1},\ldots, \mathcal{J}\dd l_{\xi^{k-1}},
\mathcal{J}\dd q_A^*f)|_{\iota_A(u_0)}\nonumber\\
&=\Pi_T(\dd l_{\widetilde{\xi}^1}+\dd h_1^{\vl},\ldots,\dd
l_{\widetilde{\xi}^{k-1}}+\dd
h_{k-1}^{\vl},\dd(\tar^*f)^{\vl})|_{\iota_A(u_0)},\label{eq:pi1}
\end{align}
with $h_i\in C^\infty(\Grd)$, $h_i|_M=0$, and $\widetilde{\xi}^i$ as
in \eqref{eq:xit} . We directly check from the definition of $\Pi_T$
that it is a linear multivector, $\Pi_T\in
\cX^k_{\mathrm{lin}}(T\Grd)$, so (see footnote \ref{foot:linear})
\begin{equation}\label{eq:vertical}
i_{\dd f_1^{\vl}}i_{\dd f_2^{\vl}}\Pi_T = 0,\;\;\; \forall f_1, f_2
\in C^\infty(\Grd).
\end{equation}
Hence the expression in \eqref{eq:pi1} agrees with
\begin{equation}\label{eq:der1}
\Pi_T(\dd l_{\widetilde{\xi}^1},\ldots,\dd
l_{\widetilde{\xi}^{k-1}},\dd(\tar^*f)^{\vl})|_{\iota_A(u_0)} =
[\Pi_T, (\tar^*f)^{\vl}](\dd l_{\widetilde{\xi}^1},\ldots,\dd
l_{\widetilde{\xi}^{k-1}})|_{\iota_A(u_0)},
\end{equation}
where $[\cdot,\cdot]$ is the Schouten bracket on
$\cX^\bullet(T\Grd)$.

Let us consider the \textit{vertical lift} operation
$\cX^\bullet(\Grd)\to \cX^\bullet(T\Grd)$, $\Pi\mapsto \Pi^{\vl}$:
in coordinates $(z^l)$ on $\Grd$, it sends the vector field
$Y=Y^l\frac{\partial}{\partial z^l}$ to $Y^{\vl} = Y^l
\frac{\partial}{\partial \dot{z}^l}$, and this is extended to a
graded algebra homomorphism of multivector fields. From the Schouten
bracket relations for vertical and tangent lifts, see e.g.
\cite{GU2}, we obtain
$$
[\Pi_T,(\tar^*f)^{\vl}] = [\Pi,\tar^*f]^{\vl} = ((\delta_\Pi
f)^r)^{\vl}.
$$
Letting $x_0=q_A(u_0)\in M$, a direct computation in coordinates
\eqref{eq:adapt} shows that
\begin{align*}
((\delta_\Pi f)^r)^{\vl}(\dd l_{\widetilde{\xi}^1},\ldots,\dd
l_{\widetilde{\xi}^{k-1}})|_{\iota_A(u_0)} = & \SP{(\delta_\Pi f)^r,
\widetilde{\xi}^1\wedge\ldots\wedge
\widetilde{\xi}^{k-1}}\Big|_{\epsilon(x_0)},\\
&= \SP{\delta_\Pi f , {\xi}^1\wedge\ldots\wedge
{\xi}^{k-1}}\Big|_{x_0},
\end{align*}
from where it follows that
$$
\pi(\dd l_{\xi^1},\ldots,\dd l_{\xi{k-1}},\dd q_A^*f)|_{u_0} =
\SP{\delta_\Pi f , {\xi}^1\wedge\ldots\wedge {\xi}^{k-1}}\Big|_{x_0}
= q_A^*\SP{\delta_\Pi f , {\xi}^1\wedge\ldots\wedge
{\xi}^{k-1}}\Big|_{u_0}.
$$
Comparing with \eqref{eq:pidelta1}, we conclude that $\delta$ (see
\eqref{eq:bijder}) and $\delta_\Pi$ agree on $C^\infty(M)$. It
remains to check that they agree on $\Gamma(A)$.

We now consider $\xi^1,\ldots,\xi^k\in \Gamma(A^*)$ and describe
$\pi(\dd l_{\xi^1},\ldots,\dd l_{\xi^k})|_{u_0}$ in terms of
$\delta_\Pi$. By \eqref{eq:pitheta} and \eqref{eq:Jb}, we have
(keeping the notation of Lemma~\ref{lem:J})
\begin{align}
\pi(\dd l_{\xi^1},\ldots,\dd l_{\xi^k})|_{u_0}& =
\Pi_T(\mathcal{J}\dd l_{\xi^1},\ldots,\mathcal{J}\dd l_{\xi^k})|_{u_0} \nonumber \\
&=\Pi_T(\dd l_{\widetilde{\xi}^1}+\dd h_1^{\vl},\ldots, \dd
l_{\widetilde{\xi}^k}+\dd h_k^{\vl})|_{\iota_A(u_0)}.
\label{eq:exp1}
\end{align}
From \eqref{eq:vertical}, we see that the expression $\Pi_T(\dd
l_{\widetilde{\xi}^1}+\dd h_1^{\vl},\ldots, \dd
l_{\widetilde{\xi}^k}+\dd h_k^{\vl})$ can be re-written as
\begin{equation}\label{eq:simp}
\Pi_T(\dd l_{\widetilde{\xi}^1},\ldots,\dd l_{\widetilde{\xi}^k}) +
\sum_{j=1}^k \Pi_T(\mathcal{J}\dd l_{\xi^1},\ldots,\mathcal{J}\dd
l_{\xi^{j-1}}, dh_j^{\vl}, \mathcal{J}\dd l_{\xi^{j+1}},\ldots,
\mathcal{J}\dd l_{\xi^k} ).
\end{equation}
We claim that, for all $j=1,\ldots,k$, we have
\begin{equation}\label{eq:vanish}
\Pi_T(\mathcal{J}\dd l_{\xi^1},\ldots,\mathcal{J}\dd l_{\xi^{j-1}},
\dd h_j^{\vl}, \mathcal{J}\dd l_{\xi^{j+1}},\ldots, \mathcal{J}\dd
l_{\xi^k}) = 0.
\end{equation}
To see that, recall from Remark~\ref{rem:sharp2} that $\Pi_T$
satisfies $ \Pi_T^\sharp \circ \Theta^{(k-1)} = (J_\Grd)^{-1}\circ
T\Pi^\sharp $, and, since $\Pi^\sharp: \oplus^{k-1}_\Grd T^*\Grd\to
T\Grd$ is a groupoid morphism (see Remark~\ref{rem:sharp1}),
$$
T\Pi^\sharp\circ (\iota_{A(T^*\Grd)})^{(k-1)} \subseteq A(T\Grd).
$$
It follows from \eqref{eq:diagJ} and the definition of $\mathcal{J}$
that $\Pi_T^\sharp\circ \mathcal{J}^{(k-1)}\subseteq
T\iota_A(TA)\subset \iota_A^*T(T\Grd)$. Relative to the adapted
coordinates $(x^j,y^d)$ in \eqref{eq:adapt}, elements in
$T\iota_A(TA)\subset \iota_A^*T(T\Grd)$ are combinations of
$\frac{\partial}{\partial x^j}$ and $\frac{\partial}{\partial
\dot{y}^d}$, whereas $\dd h_j^{\vl}|_{\iota_A(u)}$ is in the span of
$\dd y^d$. So \eqref{eq:vanish} follows, and we conclude that
\begin{equation}\label{eq:simp2}
\pi(\dd l_{\xi^1},\ldots,\dd l_{\xi^k})|_{u_0} = \Pi_T(\dd
l_{\widetilde{\xi}^1},\ldots,\dd
l_{\widetilde{\xi}^k})|_{\iota_A(u_0)}.
\end{equation}

To proceed, we observe that $\Pi_T(\dd
l_{\widetilde{\xi}^1},\ldots,\dd l_{\widetilde{\xi}^k})$ defines a
linear function on $T\Grd$, and, using \eqref{eq:linf} in
Lemma~\ref{lem:linf} (with $B=T\Grd$), we write
\begin{equation}\label{eq:eval}
\Pi_T(\dd l_{\widetilde{\xi}^1},\ldots,\dd
l_{\widetilde{\xi}^k})|_{\iota_A(u_0)} =
\Lie_{(u^r)^{\vl}}(\Pi_T(\dd l_{\widetilde{\xi}^1},\ldots,\dd
l_{\widetilde{\xi}^k}))|_{\iota_A(u_0)},
\end{equation}
where $u\in \Gamma(A)$ is such that $u(x_0)=u_0$. But
\begin{align*}\label{eq:leib}
\Lie_{(u^r)^{\vl}}(\Pi_T(\dd l_{\widetilde{\xi}^1},\ldots,\dd
l_{\widetilde{\xi}^k}))=& (\Lie_{(u^r)^{\vl}}
\Pi_T)(\dd l_{\widetilde{\xi}^1},\ldots,\dd l_{\widetilde{\xi}^k})\\
& + \sum_{j=1}^k\Pi_T(\dd l_{\widetilde{\xi}^1},\dd
l_{\widetilde{\xi}^{j-1}},\Lie_{(u^r)^{\vl}}(\dd
l_{\widetilde{\xi}^{j}}), \dd l_{\widetilde{\xi}^{j+1}},\ldots, \dd
l_{\widetilde{\xi}^{k}} )), \nonumber
\end{align*}
and note that
$$
\Lie_{(u^r)^{\vl}}(\dd l_{\widetilde{\xi}^l}) = \dd
(\widetilde{\xi}^l(u^r))^{\vl},\;\;\;
\Lie_{(u^r)^{\vl}}\Pi_T=[u^r,\Pi]^{\vl}=-((\delta_\Pi u)^r)^{\vl},
$$
where we used the Schouten-bracket relations for tangent and
vertical lifts in the second equation. One can directly check that
$$
((\delta_\Pi u)^r)^{\vl}(\dd l_{\widetilde{\xi}^1},\ldots,\dd
l_{\widetilde{\xi}^k})|_{\iota_A(u_0)}=\SP{(\delta_\Pi u)^r,
\widetilde{\xi}^1\wedge\ldots\wedge
{\widetilde{\xi}^k}}\Big|_{\epsilon(x_0)} = \SP{\delta_\Pi u,
{\xi}^1\wedge\ldots\wedge {\xi}^k}\Big|_{x_0}.
$$
Thus $\Lie_{(u^r)^{\vl}}(\Pi_T(\dd l_{\widetilde{\xi}^1},\ldots,\dd
l_{\widetilde{\xi}^k}))|_{\iota_A(u_0)}$ equals
\begin{align*}
-\SP{\delta_\Pi u, {\xi}^1\wedge\ldots\wedge {\xi}^k}\Big|_{x_0} +
\sum_{j=1}^k \Pi_T(\dd l_{\widetilde{\xi}^1},\ldots,\dd
l_{\widetilde{\xi}^{j-1}}, \dd (\widetilde{\xi}^j(u^r))^{\vl} , \dd
l_{\widetilde{\xi}^{j+1}},\ldots, \dd
l_{\widetilde{\xi}^{k}})|_{\iota_A(u_0)}.
\end{align*}
Using local coordinates $(x^j,y^d)$ as in \eqref{eq:adapt}, one can
check the identity
$$
\dd(\widetilde{\xi}^j(u^r))^{\vl}|_{\iota_A(u_0)} = \dd
(\tar^*\SP{\xi^j,u})^{\vl}|_{\iota_A(u_0)} + \dd
h_j^{\vl}|_{\iota_A(u_0)},
$$
where $h_j\in C^\infty(\Grd)$ vanishes on $M\subset \Grd$. It
follows that
\begin{align}
\Pi_T(\dd l_{\widetilde{\xi}^1},\ldots,&\dd
l_{\widetilde{\xi}^{j-1}}, \dd (\widetilde{\xi}^j(u^r))^{\vl} , \dd
l_{\widetilde{\xi}^{j+1}},\ldots, \dd
l_{\widetilde{\xi}^{k}})|_{\iota_A(u_0)}=\\ \nonumber & \Pi_T(\dd
l_{\widetilde{\xi}^1},\ldots,\dd l_{\widetilde{\xi}^{j-1}},\dd
(\tar^*\SP{\xi^j,u})^{\vl} , \dd l_{\widetilde{\xi}^{j+1}},\ldots,
\dd l_{\widetilde{\xi}^{k}})|_{\iota_A(u_0)} + \\ \nonumber &
\Pi_T(\dd l_{\widetilde{\xi}^1},\ldots,\dd
l_{\widetilde{\xi}^{j-1}}, \dd h_j^{\vl}, \dd
l_{\widetilde{\xi}^{j+1}},\ldots, \dd
l_{\widetilde{\xi}^{k}})|_{\iota_A(u_0)}.
\end{align}
Using the linearity of $\Pi_T$ (see footnote \ref{foot:linear}) and
 \eqref{eq:vanish}, we see that
\begin{align*}
\Pi_T(\dd l_{\widetilde{\xi}^1},\ldots,&\dd
l_{\widetilde{\xi}^{j-1}}, \dd h_j^{\vl} , \dd
l_{\widetilde{\xi}^{j+1}},\ldots, \dd l_{\widetilde{\xi}^{k}}) \\&=
\Pi_T(\mathcal{J}\dd l_{\widetilde{\xi}^1},\ldots,\mathcal{J}\dd
l_{\widetilde{\xi}^{j-1}}, \dd h_j^{\vl} , \mathcal{J} \dd
l_{\widetilde{\xi}^{j+1}},\ldots, \mathcal{J}\dd
l_{\widetilde{\xi}^{k}}) = 0.
\end{align*}
A direct comparison with \eqref{eq:pi1}, \eqref{eq:der1} gives that
\begin{align*}
\Pi_T(\dd l_{\widetilde{\xi}^1},\ldots,\dd
l_{\widetilde{\xi}^{j-1}},\dd (\tar^*\SP{\xi^j,u})^{\vl} , & \dd
l_{\widetilde{\xi}^{j+1}},\ldots, \dd
l_{\widetilde{\xi}^{k}})|_{\iota_A(u_0)} = \\& (-1)^{k-j} \pi(\dd
l_{\xi^1},\ldots,\widehat{\dd l_{\xi^j}},\ldots,\dd l_{\xi^k},\dd
q_A^*\SP{\xi^j,u})|_{u_0}
\end{align*}
Going back to \eqref{eq:simp2}, we finally conclude that
\begin{align*}
\pi(\dd l_{\xi^1},\ldots,\dd l_{\xi^k})|_{u_0} = & -\SP{\delta_\Pi
u,
{\xi}^1\wedge\ldots\wedge {\xi}^k}\Big|_{x_0}\\
& + \sum_{j=1}^k (-1)^{k-j} \pi(\dd l_{\xi^1},\ldots,\widehat{\dd
l_{\xi^j}},\ldots,\dd l_{\xi^k},\dd q_A^*\SP{\xi^j,u})|_{u_0}.
\end{align*}
Comparing with \eqref{eq:pidelta2}, we conclude that
$\delta=\delta_\Pi$.
\end{proof}

%%%%%%%%%%%%%%%%%%%%%%%%%%%%%%%%%%%%%%%%%%%%%%%%%%%%%%%%%%%%%%%%%%%%%%%%%%%%

\end{document}